\documentclass{amsart}

\usepackage[utf8]{inputenc} 
\usepackage{amssymb}
\usepackage{soul}
\setulcolor{red}
\sethlcolor{yellow}
\usepackage{color}

\def\rife#1{(\ref{#1})}

\newbox\barleftbox 
\newbox\barrightbox 
{\setlength\fboxsep{0pt}\setlength\fboxrule{0pt} 
 \global\setbox\barleftbox=\hbox{%
   \colorbox[gray]{.8}{\hskip 4pt\relax\vrule height3pt width 0pt}%
   \hskip4pt} 
 \dimen0=\ht\barleftbox \advance\dimen0-1pt 
 \ht\barleftbox=\dimen0 
 \dp\barleftbox=-1pt 
 \global\setbox\barrightbox=\hbox{%
   \hskip4pt 
   \colorbox[gray]{.8}{\hskip 4pt\relax\vrule height3pt width 0pt}} 
 \dimen0=\ht\barrightbox \advance\dimen0-1pt 
 \ht\barrightbox=\dimen0 
 \dp\barrightbox=-1pt 
} 

\newcommand*{\thethmname}{} 
\newtheorem{innerthm}{\thethmname} 

 

\newcommand{\leftstrip}[1]{%
   \valign{##\cr 
           \leaders\copy\barleftbox\vfill\cr 
           \vbox{\hsize\marginparwidth\advance\hsize-8pt 
                 \raggedright\sffamily\footnotesize #1}\cr 
   } 
} 
\newcommand{\rightstrip}[1]{%
   \valign{##\cr 
           \vbox{\hsize\marginparwidth\advance\hsize-8pt 
                 \raggedright\sffamily\footnotesize #1}\cr 
           \leaders\copy\barrightbox\vfill\cr 
   } 
} 

\let\oldmarginpar\marginpar 
\renewcommand\marginpar[1]{%
  \oldmarginpar[\leftstrip{#1}]{\rightstrip{#1}}}
\usepackage{amsmath}
\usepackage{amsfonts}
\usepackage{graphics}
\usepackage{color}

\newtheorem{theorem}{Theorem}[section]
\newtheorem{lemma}[theorem]{Lemma}
\newtheorem{e-proposition}[theorem]{Proposition}
\newtheorem{corollary}[theorem]{Corollary}
\newtheorem{e-definition}[theorem]{Definition}
\newtheorem{remark}{Remark}[section]
\newtheorem{example}{Example}[section]
\newtheorem{theoreme}{Th\'eor\`eme}[section]

\newtheorem{proposition}[theoreme]{Proposition}

\newtheorem{definition}[theoreme]{Definition}

\numberwithin{equation}{section}

\long\def\salta#1{\relax}

\def\la{\lambda}
\def\rn{\mathbb{R}^{N}}

\def\de{\delta}
\def\re{\mathbb{R}}

\def\liq{L^{\infty}(Q)}
\def\ga{\gamma}

\def\be{\begin{equation}}
\def\ee{\end{equation}}

\def\tku{T_{k}(u)}
\def\tkun{T_{k}(u_n)}
\def\vare{\varepsilon}

\def\capp{\text{\text{cap}}_{p}}
\def\kap#1{\text{\text{cap}}_{#1}}

\def\dive{{\rm div}}

\def\sob{W^{1,p}_{0}(\Omega)}

\newcommand{\elle}[1]{L^{#1}(\Omega)}
\newcommand{\pelle}[1]{L^{#1}(Q)}

\def\ujn{u_{n}}

\def\into{\int_{\Omega}}
\def\intq{\int_Q}

\def\liq{L^{\infty}(Q)}
\def\w-1p'{W^{-1,p'}(\Omega)}

\def\w-1pd{W^{-1,p'}(D)}
\def\pw-1p'{L^{p'}(0,T;W^{-1,p'}(\Omega))}
\def\l{\textsl{L}}

\def\dys{\displaystyle}

\def\luq{L^{1}(Q)}
\def\lp'n{(L^{p'}(\Omega))^{N}}

\newcommand{\supp}{\operatorname{supp}}

\def\lio{L^{\infty}(\Omega)}

\def\psob{L^{p}(0,T;W^{1,p}_{0}(\Omega))}

\def\huz{H^1_0 (\Omega)}
\def\eps{\varepsilon}

\def\sobl{L^2 (0,T ; H^1_0 ( \Omega ))}

\def\luo{L^{1}(\Omega)}
\def\lio{L^{\infty}(\Omega)}

\def\vep{\varepsilon}

\def\limitate#1{L^{\infty}(0,T;L^{#1}(\Omega))}

\def\duale{W^{-1,p'}(\Omega)}
\def\parelle#1{L^{#1}(Q)}
\def\vfi{\varphi}

\def\car#1{\raise2pt\hbox{$\chi$}_{#1}}

\def\adixt#1{a(t,x,#1)}

\def\lio{L^{\infty}(\Omega)}
\def\pmisure{\cM(Q)}
\def\liq{L^{\infty}(Q)}

\def\lp'n{(L^{p'}(\Omega))^{N}}

\def\huz{H^1_0 (\Omega)}

\def\sob{W^{1,p}_{0}(\Omega)}
\def\t1p0{T^{1,p}_{0}(\Omega)}
\def\w-1p'{W^{-1,p'}(\Omega)}
\def\pw-1p'{L^{p'}(0,T;W^{-1,p'}(\Omega))}
\def\psob{L^{p}(0,T;W^{1,p}_{0}(\Omega))}
\def\lil2{L^{\infty}(0,T;L^2 (\Omega))}

\def\l2h10{L^2 (0,T ; H^1_0 ( \Omega ))}

\def\m2{M^{\frac{N(p-1)}{N-1}}(\Omega)}

\def\tkun{T_k (u_n)}

\def\tkdun{T_{k,\delta}(u_n)}
\def\skdun{S_{k,\delta}(u_n)}
\def\tkdujn{T_{k,\delta}(\ujn)}
\def\skdujn{S_{k,\delta}(u_{n})}

\def\capp{{\text{\rm cap}}_{p}}

\def\l{\textsl{L}}

\def\al{\alpha}
\def\de{\delta}

\def\vare{\varepsilon}

\def\into{\int_{\Omega}}

\def\intq1{\displaystyle \int_{\Omega \times (0, 1)}}

\def\dys{\displaystyle}

\def\m{\noalign{\medskip}}

\newcommand{\RR}{\mathbb{R}}

\newcommand{\cD}{\mathcal{D}}

\newcommand{\cM}{\mathcal{M}}

\begin{document}

\title{Diffuse measures and nonlinear parabolic equations}

\author[F. Petitta]{Francesco Petitta}
\address{Francesco Petitta\hfill\break\indent 
Universitat de Valencia\hfill\break\indent 
Departamento de An\'alisis Matem\'atico\hfill\break\indent 
C/ Dr. Moliner 50\hfill\break\indent 
46100 Burjassot, Valencia, Spain}
\email{francesco.petitta@uv.es}
\author[A.~C. Ponce]{Augusto C.~Ponce} 
\address{Augusto C.~Ponce\hfill\break\indent 
Universit\'e catholique de Louvain\hfill\break\indent
Institut de Recherche en Math\'ematique et Physique\hfill\break\indent
Chemin du Cyclotron 2\hfill\break\indent
1348 Louvain-la-Neuve, Belgium}
\email{augusto.ponce@uclouvain.be}
\author[A. Porretta]{Alessio Porretta} 
\address{Alessio Porretta\hfill\break\indent
Universit\`a di Roma Tor Vergata\hfill\break\indent
Dipartimento di Matematica\hfill\break\indent
Via della ricerca scientifica 1\hfill\break\indent
00133 Roma, Italy}
\email{porretta@mat.uniroma2.it}


\thanks{The first author is partially supported by the
Spanish PNPGC project, reference MTM2008- 03176. The second author (A.C.P.) was partially supported by the Fonds de la Recherche scientifique---FNRS (Belgium) and by the Fonds sp\'eciaux de Recherche (Université catholique de Louvain)}

\begin{abstract}

Given a parabolic cylinder $Q =(0,T)\times\Omega$, where $\Omega\subset \rn$ is a bounded domain, we 
prove new properties of solutions of 
\[
u_t-\Delta_p u = \mu \quad \text{in $Q$}
\]
with Dirichlet boundary conditions, where $\mu$ is a finite Radon measure in $Q$. We first prove  a priori estimates on the $p$-parabolic capacity of level sets of $u$. We then show that diffuse measures (i.e.\@ measures which do not charge sets of zero parabolic $p$-capacity) can be strongly approximated by the measures  $\mu_k = (T_k(u))_t-\Delta_p(T_k(u))$, and  we introduce a new notion of renormalized solution based on this property. We finally apply our new approach  to prove the existence of solutions of 
$$
u_t-\Delta_{p} u + h(u)=\mu \quad \text{in $Q$,}
$$
for any function $h$ such that  $h(s)s\geq 0$ and for any diffuse measure $\mu$; when $h$ is nondecreasing we also prove uniqueness in the renormalized formulation. Extensions are given to the case of more general nonlinear operators in divergence form.
\end{abstract}

\maketitle

{\bf Key words}: {\sl parabolic capacity, measure data, nonlinear equations with absorption, renormalized solutions.}


\section{Introduction and main results}

Given a bounded domain $\Omega\subset \rn$ and $T>0$, let $Q =(0,T)\times\Omega$. We denote  by $\mathcal{M}(Q)$ the vector space of all finite Radon measures in $Q$ equipped with the norm $\|\mu\|_{\cM(Q)} = |\mu|(Q)$.
This paper is motivated by the study of  the evolution problem
\be\label{lot}
\begin{cases}
  u_t-\Delta_{p} u + h(u)=\mu & \text{in } Q,\\
 u=u_0  & \text{on}\ \{0\} \times \Omega,\\
 u=0 &\text{on}\ (0,T)\times\partial\Omega,
  \end{cases}
\ee
 where $-\Delta_pu= -{\rm div}\left( |\nabla u|^{p-2}\nabla u\right)$ is the $p$-Laplace operator, $p>1$, $\mu\in \pmisure$, $u_0 \in L^1(\Omega)$ and $h : \RR \to \RR$ is a continuous function such that $h(s)s\geq0$ for large $|s|$. It is well known (see e.g. \cite{BaPi}) that problem \eqref{lot} may not have a solution for every measure $\mu$ (unless some growth restriction is imposed on $h$). As suggested by the stationary  case (see \cite{bmp}),  if one looks for  a  \lq\lq general solvability\rq\rq\  result    (i.e.  existence of solutions of \eqref{lot} for \emph{any function $h$}),
 then it is necessary to restrict the attention to 
   the class of   measures $\mu$ which do not charge sets of zero capacity. Here, the notion of capacity which is relevant is the so-called \emph{parabolic $p$-capacity}. 

To be precise, we recall that for every $p > 1$ and every open subset
$U \subset Q$, the $p$-parabolic capacity of $U$ is given by (see \cite{pierre1,dpp})
\begin{equation}\label{1}
\capp(U)=\inf{ \Big\{\|u\|_{W} : u\in W,\ u \geq \chi_{U}\  \text{a.e. in}\  Q \Big\}},
\end{equation}
where 
\begin{equation}\label{W}
W=\big\{ u\in L^{p}(0,T;V) : u_{t}\in L^{p'}(0,T;V') \big\},
\end{equation}
being $V=\sob \cap \elle 2$ and $V'$  its dual space. As usual $W$ is endowed with the norm \begin{equation}\label{2}
\|u\|_{W} =\|u\|_{ L^{p}(0,T;V)} + \|u_{t}\|_{ L^{p'}(0,T; V')}.
\end{equation} 
The $p$-parabolic capacity $\capp$ is then extended to arbitrary Borel subsets $B \subset Q$ as
\[
\capp(B) = \inf{\Big\{ \capp(U) : B \subset U \text{ and } U \subset Q\ \text{is open}  \Big\}}.
\]

\vskip0.5em
Henceforth, we call a finite measure $\mu$ \emph{diffuse}\/ if it does not charge sets of zero $p$-parabolic capacity, i.e.\@ if $\mu(E) = 0$ for every Borel set $E \subset Q$ such that $\capp(E) = 0$.  The subspace of all diffuse measures in $Q$ will be denoted by $\cM_0(Q)$. 

\vskip0.5em
One of our goals is to prove that  \eqref{lot}  admits a solution for every diffuse measure and every $h$ satisfying the sign condition. 
In the elliptic case, such result can be proved  using the representation of   diffuse measures as elements in $\elle1+\duale$ (see \cite{BGO}). Moreover, if $h$ is increasing, then (entropy or renormalized) solutions are unique.

In the parabolic case,  the situation is more delicate.
According to a representation theorem for  diffuse measures proved in \cite{dpp}, for every $\mu \in \mathcal{M}_0(Q)$ there exist $f\in L^{1}(Q)$, $g\in
L^{p}(0,T;V)$ and $\chi\in \pw-1p'$ such that
\begin{equation}\label{cap3c1}
\mu = f + g_t + \chi \quad \text{in } \mathcal{D'}(Q).   
\end{equation}
In the same  paper, in order to deal with nonlinear equations where the source term is a  diffuse measure, the authors introduced a renormalized formulation which is based on this representation.  However, in contrast with the elliptic case, such representation, as well as the formulation suggested in \cite{dpp}, are not suitable to handle  the case of absorption terms as in \eqref{lot}. The main reason is that  a solution of 
$$
u_t-\Delta_p u=\mu=f+ \chi+g_t \quad \text{in $Q$}
$$
is meant in the sense that $v= u-g$ satisfies
$$
v_t-\Delta_p(v+g)=f+ \chi  \quad \text{in $Q$}.
$$
The same approach for problem \eqref{lot} would transform the absorption term $h(u)$ into $h(v+g)$. However, since no growth restriction  is made on $h$, this term can not be easily handled if $g$ is not bounded.

Since the decomposition \eqref{cap3c1} is not uniquely determined,  a natural question would be whether every diffuse measure can be written as \eqref{cap3c1} for some $g \in L^\infty(Q)$. 
Unfortunately the answer is \textit{no}, as we show in Example~\ref{unb} below.

\vskip0.5em
In this paper, we  overcome this obstruction by developing  a different approach to deal with diffuse measures. In this way we establish new properties of diffuse measures (related to different  types of approximations) and new results for parabolic equations (including \eqref{lot}). Both issues  are closely related as we will see later. 

As far as diffuse measures are concerned, one of the results that we prove  is that every $\mu \in \cM_0(Q)$ can be strongly approximated by measures  which admit decomposition \eqref{cap3c1} with $g\in L^\infty(Q)$. 

\begin{theorem}\label{app}
Let $\mu\in \mathcal{M}_0(Q)$. Then, for every $\vare >0$ there exists $\nu \in \cM_0(Q)$ such that 
\begin{equation}\label{app1}
\|\mu-\nu\|_{\mathcal{M}(Q)}\leq \vare \quad \text{and} \quad \nu = w_t  - \Delta_p w \quad \text{in } \cD'(Q),
\end{equation}
where $w\in \psob\cap\liq$.
\end{theorem}

Not only is this density result interesting in itself, but also the construction and the properties of the approximation $\nu$ are important. Indeed, the function $w$ is  constructed as  the truncation of  a nonlinear  potential of $\mu$ (we call in this way  a  function $u$ such that $u_t-\Delta_p u=\mu$). As a consequence, the approximation  property for the measure $\mu$ is linked to  a property of its nonlinear potential. 

The main ingredient in the proof of Theorem~\ref{app} is a capacitary estimate on the level sets of $u$. Such estimates for solutions of  parabolic equations   have an independent interest and read as follows.

\begin{theorem}\label{stimcap}
Given $\mu \in \mathcal{M}(Q)\cap\pw-1p'$ and $u_0\in \elle2$, let $u\in W$ be the 
 (unique) weak
solution of
\begin{equation}\label{appeq}
\begin{cases}
    u_t-\Delta_{p} u=\mu & \text{in}\ Q,\\
u=u_0  & \text{on}\ \{0\} \times \Omega,\\
u=0 &\text{on}\ (0,T)\times\partial\Omega.
  \end{cases}
\end{equation}
Then,
\begin{equation}\label{stim}
\capp(\{|u| > k\})\leq C\max\left\{\frac{1}{k^{\frac{1}{p}}},\frac{1}{k^\frac{1}{p'}}\right\} \quad \forall k\ge 1,
\end{equation}
where $C >0$ is a constant depending on $\|\mu\|_{\mathcal{M}(Q)}$, $\|u_0\|_{\elle1}$ and $p$.
\end{theorem}

In \eqref{stim} we have identified $u$ with its cap-quasicontinuous representative, which exists since $u\in W$  (see \cite{dpp}). In particular, the quantity $\capp(\{|u| > k\})$ is well-defined.

\vskip1em
Given a diffuse measure $\mu$, we apply Theorem~\ref{stimcap} to construct a measurable function $u : Q \to \RR$ such that the truncations $T_k(u)$ satisfy
\be\label{k1}
(T_k(u))_t-\Delta_p(T_k(u))= \mu+ \la_k \quad \text{in $Q$}
\ee
for a sequence of measures $(\la_k)$ such that
\be\label{k2}
\|\la_k\|_{\mathcal{M}(Q)}\to 0 .
\ee
We deduce in particular the strong approximation property given in Theorem~\ref{app}. On the other hand, this result also motivates an alternative formulation of the concept of renormalized solution of
\begin{equation}\label{appeq-new}
\begin{cases}
u_t-\Delta_{p} u=\mu & \text{in}\ Q,\\
u=u_0  & \text{on}\ \{0\} \times \Omega,\\
u=0 &\text{on}\ (0,T)\times\partial\Omega.
  \end{cases}
\end{equation}
in terms of properties \eqref{k1}--\eqref{k2}. Such a formulation, no more based on the decomposition \eqref{cap3c1}, can be extended to problem \eqref{lot} straightforwardly and turns out to be suitable to tackle the absorption problem.

We prove that this formulation extends the one given in \cite{BlMu} for $L^1$-data and, in case of problem \eqref{appeq-new},  is equivalent with the definition for diffuse measures given in \cite{dpp} (and therefore equivalent with the entropic formulation in \cite{DP}). This formulation we use is largely inspired by (and it is very close to) other versions of renormalized formulations in the literature as, for example, in \cite{DMOP} for elliptic equations and \cite{BCW} for conservation laws. 

We obtain in this way a new approach to solve nonlinear problems involving diffuse measures:

\begin{theorem}\label{thm-lot}
Let $\mu\in \cM_0(Q)$ and $u_0\in L^1 (\Omega)$. Let $h : \RR \to \RR$ be a continuous function satisfying
\[
h(s)s\geq 0 \quad \text{for every $|s|>L$},
\]
for some $L \ge 0$. Then, \eqref{lot} admits a renormalized solution (which is, in particular,  a distributional solution). If in addition $h$ is nondecreasing, then the renormalized solution is unique. 
\end{theorem}

The proof of Theorem~\ref{thm-lot} strongly relies on the new ingredients developed so far. 
The existence of a solution $u$ is obtained  as limit of solutions $u_n$ corresponding to a smooth approximation  $\mu_n$ of the measure $\mu$. In this procedure the difficult point is to prove the $L^1$-convergence of the lower order term $h(u_n)$. When the sequence $(\mu_n)$ is strongly (or even weakly) converging in $L^1(Q)$, this is usually deduced (see e.g. \cite{GM}) from the estimate
\be\label{equi-cap}
\int\limits_{\{|u_n|>k\}} |h(u_n)| \leq \int\limits_{\{|u_n|>k\}} |\mu_n|,
\ee
using the equi-integrability of $(\mu_n)$. In our case we extend this idea in the following sense: the level sets $\{|u_n|>k\}$ are proved to have uniformly small
capacity (Theorem~\ref{stimcap}) and the sequence $(\mu_n)$ is chosen to be equidiffuse (see Definition~\ref{def-diff} below), a concept introduced in \cite{BrePon:05a}; see also \cite{MarPon:08}. An example of such sequence is given by the convolution $\mu_n = \rho_n * \mu$, where the measure $\mu$ is diffuse (see Proposition~\ref{equi} below).
Equidiffuse sequences play the same role for the capacity, as do equi-integrable sequences for the Lebesgue measure. Therefore, coupling the capacitary estimates with the equidiffuse property of $(\mu_n)$, the right-hand side of \eqref{equi-cap} is uniformly small, implying the $L^1$-convergence of $(h(u_n))$.
Finally,  when $h$ is nondecreasing, we obtain uniqueness by proving that  the $L^1$-contraction property holds for  renormalized solutions.

\vskip1em
The article is organized in the following way.
In Section 2, we prove Theorem~\ref{stimcap}. In Section 3,  we discuss some properties of diffuse measures related to the representation \eqref{cap3c1} and we prove Theorem~\ref{app}. In Section 4, we discuss our new  renormalized formulation of problem \eqref{appeq-new}: definition,  existence and uniqueness and some properties of renormalized solutions, including a generalized version of the capacitary estimates of Theorem \ref{stimcap} (see Proposition \ref{tm}).
In Section 5 we turn our attention to problem \eqref{lot} and we prove  Theorem~\ref{thm-lot}. These results will be actually proved in the context of nonlinear monotone operators in divergence form.  We will briefly sketch in Section 6 some extension to the case of nonmonotone operators.

Part of the results of Sections 2 and 3  were announced (and proved for the case of positive measures)  in \cite{ppp1},  nevertheless we include here all the details and give  a self contained exposition for the sake of clarity.

\section{Capacitary estimates: proof of Theorem~\ref{stimcap}}

Let us first recall the following fact that will be used later. A set $E\subset Q$ is called cap-quasi open if for every $\vep>0$ there exists an open set $A_\vep\subset Q $ such that $E\subset A_\vep$ and $\capp(A_\vep\setminus E)<\vep$.  The interest in such sets  arises since,  if a function $z$ has a cap-quasi continuous representative, then its  level sets $\{ (x,t)\in Q: z(x,t)>a\}$ are cap-quasi open. It is well known that the capacity of cap-quasi open sets can be estimated as follows

\begin{lemma}\label{capquasi}
Let $E$ be a cap-quasi open subset in $Q$. Then
$$
\capp(E)\leq \inf{ \Big\{\|v\|_{W} : v\in W,\ v \geq \chi_{E}\  \text{a.e. in}\  Q \Big\}}
$$
\end{lemma}

\begin{proof}
Let $\vep>0$ and $A_\vep$ be an open set such that $E\subset A_\vep$ and $\capp(A_\vep\setminus E)<\vep$. By definition of capacity,  there exists another open set $U_\vep$ such that $A_\vep\setminus E\subset U_\vep$ and $\capp(U_\vep)<\vep$. Then, let $w_\vep\in W$ such that $w_\vep\geq \chi_{U_\vep}$ a.e. and $\|w_\vep\|_{W}\leq \vep$.  Consider the open set $A_\vep\cup U_\vep$; 
for any $v\geq \chi_E$ a.e. we have 
$v+w_\vep\geq \chi_{U_\vep\cup A_\vep}$ a.e. in $Q$, hence by \eqref{1}
$$
\capp(U_\vep\cup A_\vep)\leq \|v\|_W+ \|w_\vep\|_W\leq \|v\|_W+\vep.
$$
We deduce that
$$
\capp(E)\leq \capp(A_\vep\cup U_\vep)\leq \|v\|_W+\vep,
$$
and letting $\vep\to0$ we get
$$
\capp(E)\leq \|v\|_W.
$$
Since $v$ is arbitrary we conclude.
\end{proof}

Throughout this paper we consider a sequence of mollifiers $(\rho_n)$ such that for every $n \ge 1$, 
\begin{equation}\label{2.1a}
\rho_n \in C_c^\infty(\RR^{N+1}), \quad \supp{\rho_n} \subset B_{\frac{1}{n}}(0), \quad \rho_n \ge 0 \quad \text{and} \quad \int_{\RR^{N+1}} \rho_n = 1.
\end{equation}
Given $\mu \in \cM(Q)$, we define the convolution $\rho_n * \mu$ for every $(t, x) \in \RR \times \RR^{N}$ by
\[
(\rho_n * \mu)(t, x) = \int_Q \rho_n(t-s, x-y) \, d\mu(s, y).
\]

Before proving Theorem~\ref{stimcap}, we recall that if $u\in W$, then \(u\) is a weak solution of   \rife{appeq} if
\begin{equation}
\label{2.2}
\int_0^T \langle u_t, v\rangle\,dt+ \int_Q |\nabla u|^{p-2}\nabla u\nabla v\,dxdt = \int_0^T \langle \mu, v\rangle\,dt\quad \forall v\in W,
\end{equation}
where $\langle \cdot, \cdot \rangle$ denotes the duality between $V$ and $V'$. 

 If \(\mu \in \mathcal{M}(Q)\cap\pw-1p'\), then \eqref{2.2}  holds for every \(v \in L^{p}(0,T;V)\), and actually  for every \(r \in [0, T]\) we have
\begin{equation}
\label{2.3}
\int_0^r \langle u_t, v\rangle\,dt+ \int_0^r\int_\Omega |\nabla u|^{p-2}\nabla u\nabla v\,dxdt = \int_0^r \langle \mu, v\rangle\,dt\quad \forall v\in W,
\end{equation}
for every \(v \in L^{p}(0,T;V)\). 

A useful identity we shall use is the following: if \(u \in W\), then
\be\label{parts}
\int_s^t \langle u_t, \psi'(u)\rangle\,dt= \into \psi(u(t))\,dx - \into \psi(u(s))\,dx,
\ee
for every $s,t\in [0,T]$ and every function $\psi : \RR \to \RR$ such that $\psi'$ is Lipschitz continuous and $\psi'(0)=0$.
Indeed, since $V=\sob \cap \elle 2$, we have $L^{p'}(0,T;V')= L^{p'}(0,T;\duale)+ L^{p'}(0,T;\elle 2)$. Using the density of $C_c^\infty([0,T]\times \Omega)$ in $W$ (see e.g.~\cite[Theorem~2.11]{dpp}) and the embedding $W\subset C^0([0,T];\elle2)$, one obtains \eqref{parts}.

\begin{proof}[Proof of Theorem~\ref{stimcap}]  
We define, for any positive $k$, 
$$
T_{k}(s)=\max{\big\{{-k}, \text{min}\{k,s\} \big\}}  \quad \forall s \in \RR.
$$ 
We divide the proof into a few steps.

\vskip0.3em
\noindent
\emph{Step $1$.} Estimates of $T_k(u)$ in the space $\lil2\cap\psob$.  

\vskip0.5em

For every \(\tau \in \RR\), let 
\[
\Theta_k(\tau) =\int_{0}^{\tau} T_k(\sigma) \, d\sigma.
\]
Take \(r \in [0, T]\). Applying \eqref{2.3} with \(v = T_k(u)\) and \eqref{parts} with  \(\psi = \Theta_k\), \(s= 0\) and \(t = r\), we have
$$
\into \Theta_k (u)(r) \, dx + \int_{0}^{r} \!\! \into |\nabla \tku |^p \, dx\, dt \leq k\|\mu\|_{\mathcal{M}(Q)}+ \into \Theta_k(u_0) \, dx,
$$
Observing that  $\frac{T_k (s)^2}{2}  \leq \Theta_k(s)\leq k|s|$, $\forall s \in\mathbb{R}$, we have
\be\label{stimar}
\into \frac{[T_k (u)(r)]^2}{2} \, dx + \int_{0}^{r} \!\! \into |\nabla \tku |^p \, dx\, dt \leq k\left( \|\mu\|_{\mathcal{M}(Q)}+\|u_0\|_{\elle1}\right)
\ee
for any $r\in [0,T]$. In particular, we deduce
\begin{equation} \label{sti1}
\|T_k(u)\|_{\lil2}^{2}\leq 2kM \quad \text{and} \quad \|T_k(u)\|_{\psob}^{p}\leq k M ,
\end{equation}
where
\be\label{defM}
M=\|\mu\|_{\mathcal{M}(Q)}+\|u_0\|_{\elle1}.
\ee

\vskip0.3em
\noindent
\emph{Step $2$.} Estimates in $W$.
\vskip0.5em

In order to deduce some estimate in $W$, we use  an idea from  \cite{pierre1}.  By standard results (see \cite{l}), there exists a unique  solution $z\in \lil2\cap\psob$  of the backward problem
\begin{equation}\label{retz}
\begin{cases}
    -z_t-\Delta_{p} z=-2\Delta_p \tku & \text{in}\ Q,\\
 z=\tku  & \text{on}\ \{T\} \times \Omega,\\
 z=0 &\text{on}\ (0,T)\times\partial\Omega.
  \end{cases}
\end{equation}
Let us multiply \eqref{retz} by $z$ and integrate between $\tau$ and $T$. Using Young's inequality
we obtain
\begin{equation*}\label{sti3}
\dys\into \frac{[z (\tau)]^{2}}{2}\, dx +  \frac{1}{2} \int_\tau^T \!\! \int_\Omega |\nabla z|^p \, dx dt
\leq\dys\into \frac{\big[\tku(T)\big]^2}{2}\, dx+C \int_\tau^T \int_\Omega |\nabla T_k(u)|^p \, dx dt
\end{equation*}
for every   $\tau\in [0,T]$. Using \eqref{stimar} with $r=T$ we deduce
$$
\dys\into \frac{[z (\tau)]^{2}}{2}\, dx + \frac{1}{2} \int_\tau^T \!\! \int_\Omega |\nabla z|^p \, dx dt 
\leq\dys  C k  \left( \|\mu\|_{\mathcal{M}(Q)}+\|u_0\|_{\elle1}\right) = CkM
$$
for every $\tau\in [0,T]$. This implies
\be\label{stizp2}
\|z \|_{L^\infty(0,T;\elle2)}^2 + \|z \|_{L^p(0,T;W^{1,p}_0(\Omega))}^p
\leq\dys  C k  M.
\ee
Recall that $V=\sob\cap \elle2$; thus,
$$
\|z\|_{L^p(0,T;V)}^p\leq C\left(  \|z \|_{L^p(0,T;W^{1,p}_0(\Omega))}^p +\|z \|_{L^p(0,T;\elle2)}^p  \right).
$$
We deduce from \eqref{stizp2} that
\be\label{stiz}
\|z\|_{L^p(0,T;V)}\leq C \left[ (kM)^{\frac1p}+ (kM)^{\frac12}\right].
\ee
Moreover, the equation in \eqref{retz} implies
\begin{equation*}\label{sti4}
\|z_t\|_{\pw-1p'}\leq C\left(\|z\|_{\psob}^{p-1} +\|\tku\|_{\psob}^{p-1}\right).
\end{equation*}
Hence, using \eqref{sti1} and \eqref{stizp2} we deduce
\be\label{stizt}
\|z_t\|_{\pw-1p'}\leq C \left( kM\right)^{\frac1{p'}} .
\ee
Combining \eqref{stiz} and \eqref{stizt} we conclude that
\begin{equation}\label{dw}
\|z\|_{W} \leq C\max\{(kM)^{\frac{1}{p}},(kM)^{\frac{1}{p'}}\},
\end{equation}
where $M$ is defined in \eqref{defM}.

\vskip0.3em
\noindent
\emph{Step $3$.} Proof completed for nonnegative data. 
\vskip0.5em
Let us assume that $\mu \geq 0$ and $u_0\geq 0$; hence we have  $u_t - \Delta_p u \geq 0$, and $u \geq 0$ in $Q$.  We claim that
\be\label{kato}
(\tku)_t-\Delta_p \tku \geq 0
.
\ee
To prove \eqref{kato}, we consider the following smooth approximation of $\tku$: let us fix $\delta>0$ and define \(S_{k,\delta} : \RR \to \RR\) by
\begin{equation}\label{skd}
\dys S_{k,\delta} (s)=
\begin{cases}
1 & \text{if } |s|\leq k,\\
0&\text{if }|s| > k+\delta,\\
\mathrm{affine} & {\rm otherwise}, 
\end{cases}
\end{equation}
and finally let us denote by $T_{k,\delta} : \RR \to \RR$ the primitive function of $S_{k,\delta}$, that is
\begin{equation}\label{tkd}
T_{k,\delta}(s)=\int_{0}^{s}S_{k,\delta}(\sigma)\, d\sigma;
\end{equation}
notice that $T_{k,\delta}(s)$ converges pointwise to $T_{k}(s)$ as $\delta$ goes to zero.

Let $\varphi\in C_c^\infty(Q)$ be a nonnegative function, and take $T_{k,\delta}'(u)\varphi$ as test function in \eqref{2.2}. We obtain, using that $\mu\geq0$ and that  $T_{k,\delta}(s)$ is concave for $s\geq 0$,  
$$
\begin{array}{l}
\dys -\int_0^T \varphi_t T_{k,\delta}(u) \, dt + \int_Q |\nabla u|^{p-2}\nabla u\cdot\nabla \varphi S_{k,\delta}(u)  \, dxdt\geq 0,
\end{array}
$$  
which yields \eqref{kato} as $\delta$ goes to $0$.

Combining \eqref{retz} and \eqref{kato} we obtain
\begin{equation}\label{2.1}
-z_t-\Delta_{p} z \geq -(\tku)_t-\Delta_p \tku .
\end{equation}
Since  both $z$ and $\tku$ belong to $\psob$, a  standard comparison argument (multiply both sides of \eqref{2.1} by  $(z-\tku)^-$) allows us to conclude that $z \geq \tku$ a.e.\@ in $Q$. In particular, $z\geq k$ a.e.\@ on $\{u>k\}$. On the other hand, since $u$ belongs to $W$, it has a   unique cap-quasicontinuous representative (still denoted  by $u$),  hence  the set $\{u>k\}$ is cap-quasi open and its  capacity  can be estimated with Lemma~\ref{capquasi}. Therefore,  we get
$$
\capp(\{ u > k\}) \leq \left\|\frac{z}{k}\right\|_W.
$$
Using \eqref{dw} we obtain \eqref{stim}.

\vskip0.3em
\noindent
\emph{Step $4$.} Comparison with $\mu^+$ and $\mu^-$  when \(\mu\) is a smooth function.
\vskip0.5em
Let us consider the case where $\mu\in C^\infty(\overline Q)$. 
Then,  \(\mu^+ \in \cM(\Omega) \cap L^{p'}(0, T; W^{-1, p'}(\Omega))\) and we can consider the unique solution $v\in W$ of the problem
\begin{equation*}\label{pb+}
\begin{cases}
    v_t-\Delta_{p} v=\mu^+ & \text{in}\ Q,\\
v=u_0^+  & \text{on}\ \{0\} \times \Omega,\\
v=0 &\text{on}\ (0,T)\times\partial\Omega.
  \end{cases}
\end{equation*}
 By comparison  principle we have $v\geq u$. Using Step $3$, we deduce that there exists  a nonnegative function $z\in W$ such that 
\[
z\geq T_k(v)\geq {T_k(u)}
\]
and 
\[\label{tow}
\|z\|_W\leq C\max{ \big\{k^{\frac{1}{p}},k^{\frac{1}{p'}}\big\}},
\]
where $C=C(\|\mu\|_{\mathcal{M}(Q)}, \|u_0\|_{\elle1},p)$.
Similarly, using the solution of \eqref{appeq} with data $-\mu^-$ and $-u_0^-$,  we deduce that there exists  a nonnegative function $w\in W$ such that
\[
T_k(u)\geq -w
\]
and 
\[\label{tow-new}
\|w\|_W\leq C\max{\big \{k^{\frac{1}{p}},k^{\frac{1}{p'}}\big\}}.
\]
We have thus proved that there exist two nonnegative functions $z$, $w\in W$ such that
$$
-w\leq T_k(u)\leq z \quad \text{and} \quad \|z\|_W+\|w\|_W \leq C\max\{k^{\frac{1}{p}},k^{\frac{1}{p'}}\},
$$
where $C$ depends on $\|\mu\|_{\mathcal{M}(Q)}$, $\|u_0\|_{\elle1}$  and $p$.

\vskip0.3em
\noindent
\emph{Step $5$.} Proof completed. 
\vskip0.5em
 Let us fix $\theta \in C_c^\infty(Q)$ and set $\tilde \mu=\theta\mu$. By  standard properties of convolutions (see e.g.~\cite[Lemma~2.25]{dpp}), given a sequence of mollifiers  $(\rho_n)$ we have $\rho_n\ast\tilde \mu\in C^\infty_c(Q)$,
 \[
 \rho_n\ast\tilde \mu\to \tilde \mu \quad \text{strongly in $\pw-1p'$}
 \]
 and
 \[
 \|\rho_n\ast\tilde \mu\|_{\mathcal{M}(Q)}\leq \|\tilde \mu\|_{\mathcal{M}(Q)}\leq  \| \mu\|_{\mathcal{M}(Q)}.
 \]
Take now $(\theta_j)$ to be a  sequence of $C_c^\infty(Q)$ functions such that $\theta_j \uparrow 1$, and 
consider the solutions $u_{j,n}$ of the problem
\be\label{jn}
\begin{cases}
    (u_{j,n})_t-\Delta_{p} u_{j,n}=\rho_n\ast(\theta_j\mu) & \text{in}\ Q,\\
u_{j,n}=u_0  & \text{on}\ \{0\} \times \Omega,\\
u_{j,n}=0 &\text{on}\ (0,T)\times\partial\Omega.
  \end{cases}
\ee
As $n\to \infty$,  the sequence $(u_{j,n})$ converges in $\psob$ to  the solution $u_j$ of \eqref{appeq} with $\theta_j\mu$ as datum. Next,   as $j\to +\infty$, 
\[
u_j \to u \quad \text{in $L^\infty(0, T; L^1(\Omega))$.}
\]
This is a consequence of a  standard $L^1$-contraction argument. Indeed,  
subtracting equations \eqref{appeq} and \eqref{jn}, and 
taking $T_1(u_{j,n}-u)$ as test function, we get (note that both $u_{j,n}$ and $u$ belong to $W$),
\begin{multline*}
 \into |u_{j,n}-u|(t) \, dx \\
\leq C \| \rho_n\ast(\theta_j\mu)-\theta_j\mu\|_{\pw-1p'}\|T_1(u_{j,n}-u)\|_{\psob}\\
  + C  \into T_1(u_{j,n}-u)(\theta_j-1)\,d\mu,
\end{multline*}
which yields
\begin{multline*}
\| (u_{j,n}-u)(t)\|_{\elle1}\\
\leq C \| \rho_n\ast(\theta_j\mu)-\theta_j\mu\|_{\pw-1p'}\|T_1(u_{j,n}-u)\|_{\psob} \\
+ C \|(1-\theta_j)\mu\|_{\mathcal{M}(Q)}.
\end{multline*}
Since for $j$ fixed $u_{j,n}$ is bounded in $\psob$, as $n\to +\infty$
the first term in the right-hand side tends to \(0\), hence
$$
\| (u_{j}-u)(t)\|_{\elle1}\leq C \|(1-\theta_j)\mu\|_{\mathcal{M}(Q)}.
$$
Since the latter term tends to zero as $j\to \infty$ by dominated convergence, we deduce the convergence of  $u_j$  to $u$. 

By Step $4$, there exist nonnegative functions $z_{j,n}$ and ${w}_{j,n}$ such that
$$
- {w}_{j,n}\leq    T_k(u_{j,n})\leq z_{j,n}
$$
and
$$
\|z_{j,n}\|_W+\| {w}_{j,n}\|_W\leq C\max{\big \{k^{\frac{1}{p}},k^{\frac{1}{p'}}\big\}},
$$
where $C=C(\|\rho_n\ast(\theta_j\mu)\|_{\mathcal{M}(Q)}, \|u_0\|_{\elle1},p)$. Since 
\[
\|\rho_n\ast(\theta_j\mu)\|_{\mathcal{M}(Q)}\leq \|\mu\|_{\mathcal{M}(Q)},
\]
the constant $C$ can be chosen independently of $n$ and $j$. The sequences $(z_{j,n})$ and $({w}_{j,n})$ being bounded in $W$, they converge weakly up to  subsequences to nonnegative functions $z, {w} \in W$ and almost everywhere in $Q$. Thus,
\[\label{quasi}
- {w}\leq   T_k(u)\leq z\quad\hbox{a.e. in $Q$}
\]
and 
$$
\|z\|_W+\| {w}\|_W\leq C\max\{k^{\frac{1}{p}},k^{\frac{1}{p'}}\},
$$
where $C=C(\|\mu\|_{\mathcal{M}(Q)}, \|u_0\|_{\elle1},p)$.  
Since $u\in W$, it admits a uniquely defined  cap-quasi continuous representative, hence the sets $\{u>k\}$ and $\{u<-k\}$ are cap-quasi open. 
Using Lemma~\ref{capquasi}, we get
$$
\dys\capp(\{ |u| > k\}) \leq     \capp(\{u >k\})+   \capp(\{u < - k\})\leq \left\|\frac{z}{k}\right\|_W+ \left\|\frac{{w}}{k}\right\|_W
$$
which yields  \eqref{stim}.
\end{proof}

\vskip1em
The same argument as above still holds for more general nonlinear operators. Consider for example the problem
\begin{equation}\label{nonlinu}
\begin{cases}
    u_t-\dive(a(t,x,\nabla u))=\mu & \text{in}\ Q,\\
u=u_0  & \text{on}\ \{0\} \times \Omega,\\
u=0 &\text{on}\ (0,T)\times\partial\Omega,
  \end{cases}
\end{equation}
where  $a : Q \times \rn \to \rn$ is a Carath\'eodory function (i.e., $a(\cdot,\cdot,\xi)$
is measurable on $Q$ for every $\xi$ in $\rn$, and $a(t,x,\cdot)$ is
continuous on $\rn$ for almost every $(t,x)$ in $Q$), such that the
following holds:
\be
a(t,x,\xi) \cdot \xi \geq \al|\xi|^p,
\label{coercp}
\ee
\be
|a(t,x,\xi)| \leq \beta[b(t,x) + |\xi|^{p-1}],
\label{cont}
\ee
\be
[a(t,x,\xi) - a(t,x,\eta)] \cdot (\xi - \eta) > 0,
\label{monot}
\ee
for almost every $(t, x)$ in $Q$, for every $\xi, \eta \in \rn$, with
$\xi \neq \eta$, where $p > 1$,
$\al$ and $\beta$ are two positive constants, and
$b$ is a nonnegative function in $L^{p'}(Q)$.

\vskip1em
We obtain in a similar  way the following capacitary estimate:

\begin{theorem}\label{stimcap2}
Assume that \eqref{coercp}--\eqref{monot} hold. Given  $u_0\in \elle2$ and $\mu \in \mathcal{M}(Q)\cap\pw-1p'$, let $u\in W$ be the (unique) weak solution of \eqref{nonlinu}. Then,
$$
\capp(\{|u| > k\})\leq C\max\left\{\frac{1}{k^{\frac{1}{p}}},\frac{1}{k^\frac{1}{p'}}\right\} \quad \forall k \ge 1,
$$
where $C >0$ is a constant depending on $\|\mu\|_{\mathcal{M}(Q)}$, $\|u_0\|_{\elle1}$, $\|b\|_{L^{p'}(Q)}$, $\alpha$, $\beta$ and $p$.
\end{theorem}

The proof  runs exactly as before, replacing $\Delta_p$ with $\dive(a(x,t, \nabla (\cdot)))$ and using in the standard way the coercivity condition \eqref{coercp} (e.g. in Step 1) and the growth condition \eqref{cont} (e.g. to estimate the right hand side in Step 2).
\vskip1em
Let us stress that  both Theorem \ref{stimcap} and Theorem \ref{stimcap2} are meant to provide estimates for usual weak solutions, this is why we asked that $\mu\in \pw-1p' $ and that $u_0\in \elle2$ in the statements.  However, the estimate only depends on the norm of $\mu$ as a measure and of $u_0$ in $\elle1$. Indeed, in Section 4 we will extend this result, in a  suitable generalized form (see Proposition \ref{tm}), by considering the larger framework of renormalized solutions.

\section{Properties of diffuse measures and  the approximation result}\label{mes}

The representation result proved in \cite{dpp} states the following:  if $\mu$ is a diffuse measure, then there exist $f\in L^{1}(Q)$, $g\in
L^{p}(0,T;V)$ and $\chi\in \pw-1p'$ such that
\begin{equation}\label{cap3c1-bis}
\mu = f + g_t + \chi \quad \text{in } \mathcal{D'}(Q).   
\end{equation}
The possibility that the above decomposition holds for some $g\in \parelle\infty$ has a  special interest, as  it was also pointed out in \cite{pe2}.
In particular, one has the following counterpart.

\begin{proposition}\label{gr}
Assume that $\mu \in \cM(Q)$ satisfies \eqref{cap3c1-bis}, where $f\in\luq$, $g\in L^p (0,T; V)$ and $\chi\in \pw-1p'$. If $g\in\liq$, then $\mu$ is diffuse. 
\end{proposition}

\begin{proof}
Because of the inner regularity of $\mu$,   it suffices to prove that for any compact set $K\subset Q$ such that $\kap{p}(K)=0$, $\mu(K)=0$.
Now, if $\kap{p}(K)=0$, then by \cite[Proposition~2.14]{dpp} there exists a sequence of functions  $\psi_n \in C^\infty_c (Q)$ such that $\psi_n\geq \chi_K $ and 
$
\|\psi_n \|_W \to 0.$

Take a smooth function $\Phi: \RR\to \RR$ such that $\Phi(0)=0$, $0\leq \Phi(s)\leq 1$, $\Phi(s) = 1$ if $s\geq 1$ and $\Phi'$, $\Phi''$ are  bounded in \(\RR\). If we set $\xi_n=\Phi(\psi_n)$, then $(\xi_n)$ is  a  sequence of smooth functions such that
$\xi_n =1$ on $K$ and $0\leq \xi_n\leq 1$ in $Q$. Moreover, we have
\begin{equation}\label{zs}
\xi_n \to 0\quad \hbox{in $L^p(0,T;V)$.}
\end{equation}
Since $(\xi_n)_t= \Phi'(\psi_n)(\psi_n)_t$, for every \(\phi \in L^\infty(Q) \cap L^p(0, T; V)\) with compact support in \(Q\) we have
\[
\begin{split}
&\left| \int_Q (\xi_n)_t\,\phi\,dxdt\right| 
 \leq \|(\psi_n)_t\|_{L^{p'}(0,T;V')}\|\Phi'(\psi_n)\,\phi\|_{L^{p}(0,T;V)}\\
& \leq \|(\psi_n)_t\|_{L^{p'}(0,T;V')} \big(\|\phi\|_{L^{\infty}(Q)} \|\Phi''\|_{L^{\infty}(\RR)} \|\psi_n\|_{L^p(0, T; V)} +\|\Phi'\|_{L^{\infty}(\RR)}  \|\phi\|_{L^p(0, T; V)} \big).
\end{split}
\]
By the strong convergence of $(\psi_n)$ in $L^p(0,T;V)$ and of $((\psi_n)_t)$ in $L^{p'}(0,T;V')$,
\begin{equation}\label{zs1}
\lim_{n \to \infty}{\int_Q (\xi_n)_t\,\phi \, dxdt} = 0.
\end{equation}
Given $\vare>0$, let $\omega \subset Q$ be an open set such that 
\[
K \subset \omega \quad \text{and} \quad |\mu|(\omega\backslash K)\leq \vare,
\]
and let $\varphi$ be a cut-off function for $K$ whose support is contained in $\omega$.
We have 
$$
\int_{K} \, d\mu = \int_Q \varphi\xi_n \, d\mu - \int_{\omega\backslash K}\varphi \xi_n \, d\mu 
$$
so that
$$
|\mu(K)|\leq \left|\int_Q f\varphi\xi_n \, dxdt -\int_Q g(\varphi \xi_n)_t \, dxdt +\int_0^T \langle \chi,\varphi\xi_n\rangle \right| \, dt + \vare. 
$$
It is easy to check, using \rife{zs}, that both 
$$
\int_Q f\varphi\xi_n \, dxdt \quad \text{and} \quad \int_0^T \langle \chi,\varphi\xi_n\rangle \, dt,
$$
go to zero as $n$ goes to infinity. Moreover,  thanks to  \eqref{zs1} and  since $g\in L^p (0,T, V)\cap\liq$, we deduce that 
$$
\int_Q g  (\varphi \xi_n)_t=\int_Q g(\varphi)_t \xi_n+\int_Q g \varphi  (\xi_n)_t \to 0.
$$
Therefore, as $n\to \infty$, we get $|\mu(K)|\leq \vare$, and since $\vare$ is arbitrary  this concludes the proof.
\end{proof}

There exist diffuse measures  whose \emph{time derivative} part $g$ is essentially unbounded. The following example provides a  typical case.

\begin{example}\label{unb}
\rm
Given $0<t_0 <T$, let 
\[
\mu= \delta_{t_0 }\otimes h ,
\]
where $h \in L^1(\Omega)$. By \cite[Theorem~2.15]{dpp}, $\mu$ is diffuse. We claim that if \eqref{cap3c1-bis} holds for some $g \in L^\infty(Q)$, then 
$h \in L^\infty(\Omega)$.\\
For simplicity, let $p=2$ (the case $p\neq 2$ can be handled in a similar way by using the $p$-Laplacian)   and let $u\in L^2 (t_0,T;\huz)\cap C([t_0,T];\luo)$ be a solution of (see \cite{po}),
\begin{equation*}\label{quad}
\begin{cases}
    u_{t}-\Delta u + |\nabla u|^2 = 0 & \text{in}\ (t_0,T)\times\Omega,\\
    u=h & \text{on}\ \{t_0\} \times \Omega,\\
    u = 0 & \text{on}\ (t_0, T) \times \partial\Omega.
  \end{cases}
\end{equation*}
Denoting by $\tilde u$ the extension of $u$ in $Q$ as identically zero on $(0,t_0) \times \Omega$, then $ \tilde{u}\in\sobl$ and
\be\label{por}
    \tilde{u}_{t}-\Delta \tilde{u} + |\nabla \tilde{u}|^2 =  \mu \quad \text{in}\ Q.
\ee
Let $f\in L^{1}(Q)$, $g\in L^{p}(0,T;V)$ and $\chi\in \pw-1p'$ be such that
\[
\mu = f + g_t + \chi \quad \text{in } \mathcal{D'}(Q).   
\]
Since \rife{por} also provides  a   decomposition of $\mu$, by \cite[Lemma~2.29]{dpp} we have
\[
\tilde{u}-g \in C([0,T];\luo).
\]
Set $w = \tilde{u}-g$. Since $w\in C([0,T];\luo)$, we have for every $\varphi\in C_c^\infty(\Omega)$,
$$
\lim_{t \nearrow t_{0}}\into |w(t,x)\varphi(x)|\, dx=\into |w(t_0,x)\varphi(x)|\, dx.
$$
Since we have, for almost every  \(t \in (0, t_0)\), 
$$
\into |w(t,x)\varphi(x)|\, dx = \int_\Omega |g(t, x) \varphi(x)| \, dxdt \leq \|g\|_{\liq} \|\varphi\|_{\luo},
$$
it follows that  
$$
\into |w(t_0,x)\varphi(x)|\, dx\leq \|g\|_{\liq}\|\varphi\|_{\luo}.
$$
Therefore, $w(t_0, \cdot)\in \elle\infty$ and  
\[
\|w(t_0, \cdot)\|_{\lio}\leq\|g\|_{\liq}.
\]
On the other hand, since $\tilde u = w + g$ on $Q$,
$$
\into |\tilde u(t,x)\varphi(x)| \, dx \le \into |w(t,x)\varphi(x)| \, dx + \|g\|_{\liq}\|\varphi\|_{\luo},
$$
for all $\varphi\in C_c^\infty(\Omega)$. Hence, using the fact that $\tilde{u}=u\chi_{[t_0,T)}$ converges to $h$ as $t  \searrow t_0$, 
$$
\into |h\varphi|\, dx\leq \into|w(t_0,x)\varphi(x)| \, dx + \|g\|_{\liq}\|\varphi\|_{\luo} \leq 2\|g\|_{\liq}\|\varphi\|_{\luo},
$$
which  implies that  $h\in L^\infty(\Omega)$.
\qed
\end{example}

\vskip1em

In view of Example~\ref{unb}, Theorem~\ref{app} is actually the best one can expect, since it shows that the class of measures written as in \eqref{cap3c1-bis} with  $g$ bounded is dense (in the strong topology)  in $\cM_0(Q)$.

Before giving the proof of this result, we  point out  a few more things.  First,  it is often useful  to work with measures compactly supported on the parabolic cylinder. The next lemma provides a tool for such situations. The proof of this result can be obtained as a straightforward modification of \cite[Lemma~2.25]{dpp}.

\begin{lemma}\label{compact}
 Let $\mu \in \cM_0(Q)$. Then,  for every $\theta \in C_c^\infty(Q)$ such that $0\leq \theta\leq 1$,
 \[
 \tilde \mu=\theta\mu
 \]
 is a diffuse measure with compact support on $Q$ such that 
 $$
 \|\tilde\mu\|_{\mathcal{M}(Q)}\leq \|\mu\|_{\mathcal{M}(Q)}.
 $$
 Moreover, if $\mu= f + g_t + \chi $ is a decomposition of $\mu$ as in \eqref{cap3c1-bis} with $\chi=\dive(H)$, $H\in L^{p'}(Q)$, then 
 $$
 \tilde \mu= \tilde{f} + \tilde{g}_t + {\tilde \chi}, 
 $$
with 
\begin{align*}
\tilde{f} & =\theta f-H\cdot\nabla\theta-\theta_t g\in L^{1}(Q),\\
\tilde{g} & =\theta g\in L^{p}(0,T;V),\\
\tilde{\chi}& =\dive(\theta H)\in \pw-1p'. 
\end{align*}
In particular, if $g\in L^{\infty}(Q)$, then $\tilde g \in L^{\infty}(Q)$.
\end{lemma}   

\medskip
We will also need an important property enjoyed by the convolution of diffuse measures. We first recall the following definition  (see \cite{BrePon:05a} and also \cite{MarPon:08}):

\smallskip
\begin{e-definition}\label{def-diff}
A sequence of measures $(\mu_n)$ in $Q$ is \emph{equidiffuse} if for every $\vare>0$ there exists $\eta>0$ such that
for every Borel measurable set \(E \subset Q\),
\[
\capp(E)<\eta \quad \Longrightarrow \quad |\mu_n|(E)<\vare \quad \forall n\geq 1. 
\]  
\end{e-definition}
\smallskip

\begin{proposition}\label{equi}
If $\mu \in \cM_0(Q)$, then the sequence $(\rho_n \ast\mu)$ is equidiffuse.  
\end{proposition}

\begin{proof}
It suffices to establish the result when $\mu \geq 0$; in the general case we can apply the conclusion to the positive and the negative parts of $\mu$. \\
Assume by contradiction that $(\rho_n\ast\mu)$ is not equidiffuse. 
Passing to a subsequence if necessary, there exist $\varepsilon > 0$ and a sequence $(E_n)$ of Borel subsets of $Q$ such that
\[
\capp{(E_n)} \le \frac{1}{n} \quad \text{and} \quad \int_{E_n} \rho_n * \mu \ge \varepsilon
\]
for every $n \ge 1$. By definition of capacity, there exists an open subset $\omega_n \subset Q$ such that
\[
E_n \subset \omega_n \quad \text{and} \quad \capp{(\omega_n)} \le \frac{2}{n}.
\]
Let $(\zeta_n)$ be a sequence in $W$ such that
\begin{equation*}\label{3.0}
\zeta_n \geq \chi_{\omega_n} \; \text{a.e.}  \quad \text{and} \quad \|\zeta_n\|_W \le \frac{3}{n}. 
\end{equation*}

\medskip
Let $A \Subset Q$ be an open set such that $\mu(Q \setminus A) < \frac{\varepsilon}{2}$ and let $\varphi \in C_c^\infty(Q)$ be such that
\[
0 \le \varphi \le 1 \; \text{in $Q$} \quad \text{and} \quad \varphi = 1  \; \text{on $A$}.
\]
Since $\varphi$ and $\zeta_n$ are nonnegative and $T_1$ is concave on $\RR^+$,
\[
T_1(\zeta_n) \le T_1(\varphi\zeta_n) +  T_1((1-\varphi)\zeta_n).
\]
Given a sequence $(\rho_n)$ of mollifiers satisfying \eqref{2.1a}, for every $n \ge 1$, we have
\[
\begin{split}
\varepsilon  \le \int_{\omega_n} \rho_n \ast \mu \, dxdt 
& \le \int_Q T_1(\zeta_n) (\rho_n \ast \mu) \, dxdt \\
& \le \int_Q T_1(\varphi\zeta_n) (\rho_n \ast \mu) \, dxdt + \int_Q T_1((1 - \varphi)\zeta_n) (\rho_n \ast \mu) \, dxdt\\
& \le \int_Q \check\rho_n \ast T_1(\varphi\zeta_n)\, d\mu + \int_{Q \setminus A}  \rho_n \ast \mu \, dxdt.
\end{split}
\]
 Since $ A$ is open, we have (see e.g. \cite[Section 1.9, Thm 1]{EG})
\begin{equation}\label{3.1a}
 \limsup\limits_{n\to \infty} \int_{Q \setminus A}  \rho_n \ast \mu \, dxdt\leq  \mu(Q \setminus A).
\end{equation}
 We now show that
\begin{equation}\label{3.1}
\lim_{n \to \infty} \int_Q \check\rho_n\ast T_1(\varphi \zeta_n) \, d\mu \to 0.
\end{equation}
Indeed, we have $\|\varphi\zeta_n\|_W \to 0$ and since $\varphi$ has compact support,
\[
\|\check\rho_n \ast(\varphi\zeta_n)\|_W \to 0.
\]
Passing to a subsequence, there exists a Borel set $F \subset Q$ such that $\capp{(F)} = 0$ and
\[
\check\rho_n \ast(\varphi\zeta_n)(x) \to 0 \quad \forall x \in Q \setminus F.
\]
Since the measure $\mu$ is diffuse, we deduce that $(\check\rho_n \ast(\varphi\zeta_n))$ converges a.e.\@ with respect to $\mu$. Since \(\varphi\zeta_n\) is nonnegative,
\[
0 \le \check\rho_n \ast T_1(\varphi\zeta_n) \le  \check\rho_n \ast\varphi\zeta_n.
\]
Thus,
\[
\check\rho_n \ast T_1(\varphi\zeta_n) \to 0 \quad \text{$\mu$-a.e.}
\]
Assertion \eqref{3.1} is now a consequence of the dominated convergence theorem.

By \eqref{3.1a}--\eqref{3.1}, we deduce that
\[
\varepsilon \le  \mu(Q \setminus A).
\]
This contradicts our choice of $A$. Therefore, the sequence $(\rho_n \ast \mu)$ is equidiffuse.
\end{proof}

We now present the
 
\begin{proof}[Proof of Theorem~\ref{app}]  
Let $\mu_n = \rho_n * \mu$, where $(\rho_n)$ is a sequence of mollifiers satisfying \eqref{2.1a}. We denote by $u_n$ the solution of \eqref{appeq} with data $\mu_n$ and $u_0=0$. 
For $k > 0$ and $\delta>0$, consider
the functions $S_{k,\delta}$ and $T_{k,\delta}$ given by \eqref{skd} and \eqref{tkd}, respectively.
Recall that $T_{k,\delta}$ converges pointwise to $T_{k}$ as $\delta \to 0$.\\
Given $\varphi \in C_c^\infty(Q)$, multiply the equation solved by $u_n$ by $\skdun \varphi$. We then have
\begin{multline}\label{equa}
 \dys 
 \tkdun_t  -\dive\left(  \skdun|\nabla u_n|^{p-2}\nabla u_n\right) \\ \dys 
 = \skdun\mu_n +\frac{1}{\delta}|\nabla u_n|^p  {\rm sign}(u_n)\chi_{\{k\leq |u_n|<k+\delta\}} \quad \text{in}\ \mathcal{D}'(Q).
\end{multline}
Using $(1-\skdun)  {\rm sign}(u_n)$ as test function in the equation \eqref{appeq} for $u_n$ we obtain 
\begin{equation}\label{esta}
\frac{1}{\delta} \underset{\{k\leq |u_n|<k+\delta\}}{\int}|\nabla u_n|^p\leq \int_{Q} \big|1-\skdun \big|| \mu_n| \leq \underset{\{|u_n|>k\}}{\int} |\mu_n|.
\end{equation}
Let us set
\[
\nu_{n}^{k} = (\tkun)_t -\Delta_{p}(\tkun).
\]
Thanks to \rife{equa}, the right-hand side of \eqref{equa} remains bounded in $L^{1}(Q)$ as $\delta \to 0$, then we deduce that $\nu_n^k$ is a finite measure in $Q$ and
\[
\begin{split}
\int_Q |\nu_n^k| 
& \le \liminf_{\delta \to 0}{\bigg\{\int_Q |\skdun\mu_n| + \frac{1}{\delta} \underset{\{k\leq |u_n|<k+\delta\}}{\int}|\nabla u_n|^p \bigg\}}\\
& \le \underset{\{|u_n|\le k\}}{\int} |\mu_n| + \underset{\{|u_n|>k\}}{\int} |\mu_n| = \|\mu_n\|_{\cM(Q)} \le \|\mu\|_{\cM(Q)}.
\end{split}
\]
In particular, $(\nu_n^k)$ remains uniformly bounded in $\cM(Q)$ as $n \to \infty$. 

We now show that 
\begin{equation}
\label{3.12}
\nu_n^k \overset{*}{\rightharpoonup} \nu^k \quad \text{weakly\(^*\) in \(\cM(Q)\)},
\end{equation}
where 
\[
\nu^k= (\tku)_t -\Delta_{p}(\tku).
\] 
To this purpose, we recall that by classical results on parabolic equations with measure data (see e.g.~Proposition~\ref{pro} below), there exists a function $u\in L^1(Q)$ such that (taking a subsequence if necessary) $u_n\to u$ and $\nabla u_n \to \nabla u$ a.e.\@ on $Q$. In particular,  since $|\nabla T_k(u_n)|^{p-2}\nabla T_k(u_n)$ is bounded in $L^{p'}(Q)$ and almost everywhere converges to $|\nabla T_k(u)|^{p-2}\nabla T_k(u)$, then it weakly converges to the same limit in $L^{p'}(Q)$. Therefore, we have
$$
\tkun \rightharpoonup \tku \quad \text{weakly in $\psob$} \quad 
$$
and
$$
\Delta_p \tkun \rightharpoonup \Delta_p \tku \quad
\text{weakly in $\pw-1p'$}
$$
as $n \to \infty$. Together with the fact that $(\nu_n^k)$ is uniformly bounded in $\cM(Q)$, this implies \eqref{3.12}. Moreover, since $T_k(u) \in \psob \cap L^\infty(Q)$, it follows from Proposition~\ref{gr} that $\nu^k$ is a diffuse measure.

By \eqref{equa}--\eqref{esta} we also have
\begin{equation}\label{est2}
\begin{split}
\dys \int_{Q}|\nu_{n}^{k}-\mu_n |
&\dys \leq \liminf_{\delta \to 0}\int_{Q}\left|\skdun\mu_{n} +\frac{|\nabla u_n|^{p}}{\delta}{\rm sign}(u_n)\chi_{\{k\leq |u_n|<k+\delta\}}-\mu_n\right|\\
& \dys \leq \underset{\{|u_n| >k\}}{\int}| \mu_n| + \limsup_{\delta \to 0}{\underset{\{k\leq |u_n|<k+\delta\}}{\int} \frac{|\nabla u_n|^p}{\delta} } \leq 2 \underset{\{|u_n| >k\}}{\int}|\mu_n|.
\end{split}
\end{equation}
Recall that by Proposition~\ref{equi} the sequence $(\mu_n)$ is equidiffuse. Applying Theorem~\ref{stimcap} we can fix $k > 0$ sufficiently large (depending only on $\eps > 0$) so that the right-hand side of \eqref{est2} is $\leq \eps$, $\forall n \geq 1$.
From \eqref{est2} and the lower semicontinuity of the norm with respect to the weak$^*$ convergence,  we obtain
\begin{equation*}
\|\nu^k -\mu\|_{\mathcal{M}(Q)}\leq\liminf_{n \to \infty}\int_{Q}|\nu_{n}^{k}-\mu_n |\leq 2\liminf_{n \to \infty} \underset{\{|u_n| >k\}}{\int} |\mu_n| \leq \eps.
\end{equation*}
This concludes the proof of Theorem~\ref{app}.
\end{proof}

\section{Renormalized formulation}

In this section, we come back to the construction of the approximation in the proof of Theorem~\ref{app} and we  develop that idea in connection with  a renormalized formulation for solutions of \eqref{appeq-new} when $\mu$ is a diffuse measure.  Because of the intrinsic interest of renormalized formulations with measure data, we deal with the more general initial boundary value problem
\begin{equation}\label{renpb}
\begin{cases}
    u_t-\dive(\adixt {\nabla u} )=\mu & \text{in}\ Q,\\
    u=u_0  & \text{on}\ \{0\} \times\Omega,\\
 u=0 &\text{on}\ (0,T)\times\partial\Omega.
  \end{cases}
\end{equation} 
 In all the following, we assume that $a : Q \times \RR^N \to \RR^N$ satisfies \eqref{coercp}--\eqref{monot}, that 
$$
u_0\in \elle1
$$ 
and that $\mu$ is a diffuse measure, i.e.
$$
\mu\in\cM_0(Q).
$$
 A notion of renormalized solution for problem \eqref{renpb} when $\mu$ is a diffuse measure was introduced in \cite{dpp} and in the same paper the existence and uniqueness of such a  solution is proved. In \cite{DP}  a similar notion of entropy solution is also defined, and proved to be equivalent to that of renormalized solution.  Our goal is to give here a new   definition which, in contrast with the previous ones, is 
not formulated in terms of  a decomposition of $\mu$ like in \eqref{cap3c1}. The next definition is certainly closer to the one used for conservation laws in \cite{BCW} and to one of the existing formulations in the elliptic case (see \cite{DalMal:97,DMOP}). 

\begin{definition}\label{rendef}
Let $\mu\in \cM_0(Q)$. A function $u\in L^1(Q)$ is a renormalized solution of problem \eqref{renpb} if $T_k(u)\in \psob$ for every $k>0$ and if there exists a  sequence $(\la_k)$ in $\pmisure$ such that
\be\label{limnu}
\lim\limits_{k\to \infty} \| \la_k\|_{\pmisure} =0,
\ee
and
\begin{multline}\label{req}
-\int_Q T_k(u)\vfi_t\,dxdt +\int_Q \adixt {\nabla T_k(u)}\nabla \vfi\,dxdt
\\
=\int_Q \vfi\,d\mu+\int_Q \vfi\,d\la_k+\into T_k(u_0)\vfi(0, x)\,dx 
\end{multline}
 for every \(k > 0\) and $\vfi\in C^\infty_c([0,T)\times \Omega)$.
\end{definition}

\begin{remark}
\label{remark4.1}
\rm
It is straightforward to check that, by approximation, one can take  as test function in \eqref{req} any $\vfi\in W^{1,\infty}(Q)$ such that $\varphi = 0$ on $(\{T\} \times \Omega) \cup ((0,T)\times \partial \Omega)$.
\end{remark}

Some considerations are in order concerning 
Definition~\ref{rendef}. First of all, observe that \eqref{req} implies that $(T_k(u))_t-\dive(\adixt{\nabla T_k(u))}$ is a finite measure and
$$
(T_k(u))_t-\dive(\adixt{\nabla T_k(u)})=\mu+\la_k\qquad \hbox{in $\pmisure$.}
$$
This provides a decomposition of the measure \(\mu+\la_k\) of the form \eqref{cap3c1-bis} with \(g \in L^\infty(Q)\).
In view of Proposition~\ref{gr}, the left hand side is a diffuse measure. Since $\mu$ itself is  diffuse, the consequence is that the measures $\la_k$ are diffuse. Moreover, condition \eqref{limnu} implies that the left hand side is  a strong approximation of $\mu$. In particular, \emph{the existence of a  renormalized solution in the sense of Definition~\ref{rendef} implies as a corollary the statement of Theorem~\ref{app}}.  

Finally, since $T_k(u)\in L^p(0, T; V)$, we have $(T_k(u))_t\in W'$\footnote{indeed, for every $v \in L^p(0, T; V)$, we have $v_t \in
W'$ in the following sense:
$$
\langle v_t, \vfi\rangle_{(W',W)}= -\int_0^T\langle v,\vfi_t \rangle_{(V,V')}\,dt\qquad \hbox{for every $\vfi\in W$}
$$
and since
$$
\left| \int_0^T\langle v,\vfi_t \rangle_{(V,V')}\,dt\right| \leq \|v\|_{L^p(0, T; V )}\|\vfi_t\|_{L^{p'}(0, T; V')}\leq \|v\|_{L^p(0, T;V )}\|\vfi\|_{W}
$$
we have $v_t\in W'$ and $\|v_t\|_{W'} \leq \|v\|_{L^p(0, T;V )}$.
}
and, due to assumption \rife{cont}, $\adixt{\nabla T_k(u)}\in L^{p'}(Q)$. Therefore, we conclude that
$$
(T_k(u))_t-\dive(\adixt{\nabla T_k(u)})\in W'\cap \pmisure.
$$

Another important fact is that we can recover from equation \eqref{req} the standard estimates known for nonlinear potentials. In order to prove such estimates  and further properties of the renormalized formulation, we will need a  few technical ingredients.

First of all, recall that any $z\in W$ admits a  unique cap-quasi continuous representative; henceforth, by 
identifying $z$ with its cap-quasi continuous representative, the integral 
$\int_Q z\,d\mu$ is well defined for every diffuse measure $\mu$ and every $z\in W$.  
Unfortunately, 
given \(z \in W\), a smooth truncation of \(z\) need not belong to \(W\). In this case, one is led  to consider the larger space
\begin{equation*}\label{spaceS}
S=\{ z\in L^p (0,T;V); z_t \in L^{p'}(0,T;\w-1p' ) +  \luq \}
\end{equation*}
endowed with its norm
$$
\| z\|_S= \|z\|_{L^p (0,T;V)}+ \|z_t\|_{L^{p'}(0,T;\w-1p' ) +  \luq }.
$$
Indeed,

\begin{lemma}
\label{tech1}
For every $v\in W$ and every function $\Phi\in C^2(\RR)$ such that $\Phi'$ and $\Phi''$ are bounded, we have $\Phi(v)\in S$ and the application $v\mapsto \Phi (v)$ is continuous from $W$ to $S$. Moreover, if $v_t=\dive(G)+g$ with $G\in L^{p'}(Q)$, $g\in L^{p'}(0,T;\elle 2)$, then
we have 
\be\label{cap-tec1}
(\Phi(v))_t=\dive(\Phi'(v)G)+ \Phi'(v)g - \Phi''(v)G\cdot \nabla v \quad \hbox{in ${\mathcal D}'(Q)$.}
\ee
 
\end{lemma}

\begin{proof}
It follows from \cite[Theorem~2.11]{dpp} that $C^\infty_c([0,T]\times \Omega)$ is dense in $W$. Then there exists a sequence of smooth functions $(v_n)$  converging to $v$ in $W$. Let $v_t=\dive(G)+g$ with $G\in L^{p'}(Q)$, $g\in L^{p'}(0,T;\elle 2)$; the convergence of $(v_n)$ to $v$ in $W$ implies that $v_n\to v$ in $L^p(0,T;V)$ and
there exist \(G_n \in L^{p'}(Q)\) and \(g_n \in L^{p'}(0,T;\elle 2)\) such that $(v_n)_t=\dive(G_n)+ g_n$ and  $G_n\to G$ in $L^{p'}(Q)$ and $g_n\to g$ in $ L^{p'}(0,T;\elle 2)$. Since the equality  \rife{cap-tec1} is true for  $v_n$, passing to the limit (which is possible thanks to the properties of $\Phi$) we recover \rife{cap-tec1} for $v$. The continuity of $T$ is established in a similar way using identity \eqref{cap-tec1}.
\end{proof}

Bounded functions in \(S\) satisfy a capacitary estimate for the parabolic capacity; we refer the reader to \cite[Theorem~3 and  Lemma~2]{pe2} for a proof of the following 

\begin{lemma}
\label{tech2}
If $z\in S\cap L^\infty(Q)$, then \(z\) admits a unique cap-quasi continuous representative. Moreover, we have
$$
\capp(\{|z|>k\})\leq \frac Ck \max{\left\{  [z]^{\frac1p}\,,\,[z]^{\frac1{p'}}\right\}},
$$
where 
\begin{multline*}\label{cap-tec2}
[z]= \inf\Big\{ \|z\|_{L^p(0,T;\sob )}^p+ \|(z_t)_1\|_{L^{p'}(0,T;\duale)}^{p'}
\\
+ \|z\|_\infty \, \|(z_t)_2\|_{L^1(Q)}+ \|z\|_{\limitate 2}^2 \Big\}
\end{multline*}
and the infimum is taken over all decomposition of $z_t= (z_t)_1+(z_t)_2$ with $(z_t)_1\in L^{p'}(0,T;\duale)$ and $(z_t)_2\in L^1(Q)$.
\end{lemma}

We obtain in particular the following corollary which allows one to pass from an inequality almost everywhere to an inequality cap-quasieverywhere.

\begin{corollary}
\label{tech4}
If $v\in W$ and $v\leq M$ almost everywhere in $Q$, then $v\leq M$ cap-quasi everywhere in $Q$. 
\end{corollary}

\begin{proof}
Let us take a bounded, nondecreasing function $\Phi\in C^2(\RR)$ such that \(\Phi', \Phi''\) are bounded,  $\Phi(s)\equiv 0$ if $s\leq 0$ and $\Phi(s)>0$ if $s>0$. By Lemma~\ref{tech1} we deduce that $\Phi(v-M)\in S\cap L^\infty(Q)$. Since $v-M\leq 0$ almost everywhere, then $\Phi(v-M) = 0$ almost everywhere. It follows from Lemma~\ref{tech2} that the unique cap-quasicontinuous representative of $\Phi(v-M)$ is the function identically zero, that is $\Phi(v-M)=0$ cap-quasi everywhere. Therefore $v\leq M$ cap-quasi everywhere.
\end{proof}

We now study the pointwise convergence of sequences in \(S\): 

\begin{lemma}
\label{tech3}
For every bounded sequence $(z_n)$ in $S\cap L^\infty(Q)$, if \(z_n \to z\) in \(S\),
then there exists a subsequence $(z_{n_k})$ converging to $z$ cap-quasi everywhere.
\end{lemma}

\begin{proof}
This is  a classical argument; we present it here for the convenience of the reader. Since \((z_n)\) is bounded in \(L^\infty(Q)\) and \(z_n \to z\) in \(S\), we have (with the notations of Lemma~\ref{tech2})
\[
[z_n - z] \to 0.
\]
Take a subsequence $(z_{n_k})$ such that
$$
\sum\limits_{k=1}^\infty  2^k  \max{\big\{ [z_{n_k}-z]^{\frac1p}\,,\,[z_{n_k}-z]^{\frac1{p'}}\big\}}
<\infty\,.
$$
Define the sets
$$
E_k= \big \{ (x,t)\in Q\,:\, | z_{n_k}- z |>2^{-k}\big \}, \quad F_m=\bigcup\limits_{k=m}^\infty E_k,\quad F_\infty= \bigcap\limits_{m=0}^\infty F_m.
$$
By the subadditivity of the parabolic capacity and by the capacitary estimate from Lemma~\ref{tech2}, we have
$$
\capp(F_m)\leq \sum\limits_{k=m}^\infty \capp(E_k)\leq C \sum\limits_{k=m}^\infty 2^k \max{\big\{  [z_{n_k}-z]^{\frac1p}, [z_{n_k}-z]^{\frac1{p'}}\big\}}.
$$
Hence $\capp(F_m)\to 0$, which implies $\capp(F_\infty)=0$. Since $z_{n_k}(x)\to z(x)$ for every $x\in Q\setminus F_\infty$, the conclusion follows.
\end{proof}

We can now show how to extend the class of test functions in \rife{req}.

\begin{proposition}\label{testW} If $u$ is a renormalized solution of \eqref{renpb}, then
\begin{multline}\label{req2}
- \int_0^T \langle T_k(u), v_t \rangle \,dt +\int_Q \adixt {\nabla T_k(u)}\nabla v\,dxdt
\\
= \int_Q  v\,d\mu+\int_Q   v\,d\la_k+\into T_k(u_0)v(0, x)\,dx 
\end{multline}
for every $v\in W\cap\parelle\infty$ such that $v=0$ on $\{T\} \times \Omega$.
\end{proposition}

\begin{proof}
Since $C_c^\infty([0,T]\times \Omega)$ is dense in $W$, there exists  a sequence $(v_n)$ in $C^\infty_c([0,T]\times \Omega)$ converging to $v$ in $W$. Let $M$ such that $\|v\|_{\parelle\infty}\leq M$, and take a function  $\Phi: \RR\to \RR$ which is  $C^\infty$ and  such that $\Phi(s)=s $ when $|s|<2M$ and $\Phi'$ has compact support.  By Lemma~\ref{tech1}, we have that $\Phi(v_n)\to \Phi(v)=v$ in $S$.
Let us  then  call $w_n=\Phi(v_n)$.
 Note that \(w_n\) is not an admissible function since it need not vanish on \(\{T\} \times \Omega\). For this purpose, choose 
 $\vfi=w_n\xi$ in \eqref{req} where $\xi\in W^{1,\infty}(0,T)$ and is compactly supported in $[0,T)$;  \(\varphi\) can be used as a test function in view of Remark~\ref{remark4.1}. We get  
\begin{multline*}
-\int_Q T_k(u)w_n\xi_t\, dxdt -\int_Q T_k(u)\xi(w_n)_t\,dxdt +\int_Q \adixt {\nabla T_k(u)}\nabla w_n\xi\,dxdt
\\
= \int_Q w_n\xi\,d\mu+\int_Q w_n\xi\,d\la_k+\into T_k(u_0)w_n(0, x)\xi(0)\,dx .
\end{multline*}
Since $w_n\to v$ in $S$ and $(w_n)$ is uniformly bounded, we can  pass to the limit in the left hand side. Moreover, since $v_n\to v$ in $W$, there exists a subsequence such that $ v_n\to  v$ cap-quasi everywhere, hence $\mu$-a.e. and $\la_k$-a.e., since the two measures are diffuse. Being $\Phi$ smooth, we have $w_n\to \Phi( v)= v$ 
$\mu$-a.e. and $\la_k$-a.e. Since $(w_n)$ is uniformly bounded we can pass to the limit in the right hand side by dominated convergence, obtaining
\begin{multline*}
-\int_Q T_k(u)v\xi_t\, dxdt -   \int_0^T \langle T_k(u)\xi, v_t \rangle \,dt +\int_Q \adixt {\nabla T_k(u)}\nabla v \, \xi\,dxdt
\\
= \int_Q  v\xi\,d\mu+\int_Q  v\xi\,d\la_k+\into T_k(u_0)v(0, x)\xi(0)\,dx .
\end{multline*}
 Given \(\eps \in (0, T)\), we now apply this identity with the function \(\xi_\eps : [0, T] \to \RR\) given by $\xi_\vep(t)= 1-\frac{(t-T+\vep)^+}\vep$; in particular $\xi_\vep=1$ in $(0,T-\vep)$, $\xi_\vep(T)=0$ and $\xi_\vep\to 1$. Since $v\in C^0([0,T];\elle2)$ and $v(T)=0$, we have
$$
\left| \int_Q T_k(u)v(\xi_\vep)_t \, dxdt\right|  \leq \frac1\vep \int_{T-\vep}^T \into | T_k(u)v| \, dxdt\leq k \frac1\vep \int_{T-\vep}^T \|v(t)\|_{\elle1} \, dt \to 0.
$$
By dominated convergence we can pass to the limit in all other terms, hence we deduce \eqref{req2}.
\end{proof}

One of  the main roles of Proposition ~\ref{testW} is to give us a way to use test functions of the form $\psi(u)$.
A difficulty arises from the fact that $u$ (and also $T_k(u)$) may not have a cap-quasi continuous representative. In order to overcome this point, our strategy will be  to use approximations through so-called Steklov time-averages for functions 
 \(z : (0, T) \times \Omega \to \RR\). More precisely, given \(\eps \in (0, T)\), define for every \(h \in (0, \eps)\) the functions \(z_h : (0, T-\eps) \times \Omega \to \RR\) and  \(z_{-h} : [\eps, T] \times \Omega \to \RR\) by
\be\label{stek}
  z_h(t,x)= \frac1h \int_t^{t+h} z(s, x) \, ds \quad \text{and} \quad z_{-h}(t,x)= \frac1h \int_{t-h}^t z(s,x) \, ds.
\ee

In the context of parabolic equations we can deal with Steklov averages as follows. Given \(w : Q \to \RR\), if \(\varphi\) is a function with compact support in \([0, T) \times \Omega\), then \(w_h \varphi\) has a meaning in \(Q\) even though the function \(w_h\) is not defined on \([T-h, T] \times \Omega\). Concerning \(w_{-h} \varphi\), one can consider this function on the parabolic cylinder $(\epsilon, T) \times \Omega$ for every \(h \in (0, \epsilon)\). Once we get the desired estimates independently of the parameter \(h\), we can let \(h \to 0\) and then \(\epsilon \to 0\). Another approach consists in taking a sequence of smooth functions \((w_{0, j})\) converging in \(L^1(\Omega)\) to \(w(0, \cdot)\) and for each \(j\) we take an extension of \(w(t, \cdot)\) as \(w_{0, j}\) if \(t < 0\); once we get uniform estimates with respect to \(j\), we can let \(j \to \infty\).

Observe that, for every $z\in L^p(0,T; V )$, we have $z_h\in W$, hence $z_h$ admits a cap-quasi continuous representative. In addition, whenever $z\in L^\infty(Q)$, we have $|z_h|\leq \|z\|_\infty$ q.e. (i.e.~except of a set of zero capacity). Indeed, for any $M$ such that $|z|\leq M$ a.e., we also have $|z_h|\leq M$ a.e. and since $z_h\in W$  we deduce from Corollary~\ref{tech4}  that $|z_h|\leq M$ q.e.\@ as well. 
Some further property of the Steklov averages with respect to capacity will be useful. For instance,

\begin{lemma}\label{cap-stek}
Let $z\in S$. Then, for every $\psi\in C_c^\infty(0,T)$,
\[
z_h \psi \to z \psi \quad \text{in S}.
\]
If in addition \(z \in L^\infty(Q)\), then for every sequence \((h_n)\) of positive numbers converging to \(0\), there exists a subsequence $(h_{n_k})$ such that 
\[
{z_{h_{n_k}}} \to  z \quad \text{q.e. in $Q$.}
\]
\end{lemma}

\begin{proof}
If $\psi\in C_c^\infty(0,T)$, then $z_h \psi \in S$. In particular, if $z_t= \dive(G)+g$, with $G\in L^{p'}(Q)$, $g\in L^1(Q)$, since $\psi$ has compact support in $(0,T)$ we have 
\[
(z_h \psi)_t= \dive(G_h\psi)+ g_h\psi + z_h \psi_t
\]
where $G_h$ and $g_h$ are Steklov averages of $G$ and $g$, respectively.  Since $G_h\to G$ in $L^{p'}(Q)$, $g_h\to g$ in $L^1(Q)$ and $z_h\to z$ in $L^1(Q)$, we have that $(z_h\psi)_t $ converges to $(z\psi)_t$ in $L^{p'}(0,T;\duale)+L^1(Q)$; on the other hand  $z_h$ converges to $z$ in $L^p(0,T;V)$ so that $z_h\psi\to z\psi $ in $L^p(0,T;V)$. Therefore, we conclude that $z_h\psi \to z\psi$ in $S$. If in addition  $z \in L^\infty(Q)$, then $(z_{h_n} \psi)$ is bounded in $S\cap L^\infty(Q)$ and,
from Lemma~\ref{tech3},  we conclude that  $z_{h_n}\psi$ admits a subsequence converging cap-quasi everywhere. We now take $\psi_j\in C_c^\infty(0,T)$ such that $\psi_j = 1$ in $(\frac 1j,T-\frac1j)$; there exists a subsequence $(h_{n_j})$ and a set $F_j \subset Q$ such that $\capp(F_j)=0$ and $z_{h_{n_j}}\to z$ in $\big( (\frac 1j,T-\frac1j)\times \Omega\big) \setminus F_j$. Using a diagonal argument we can construct  a subsequence $(h_{n_k})$ such that $z_{h_{n_k}}\to z$ in $Q\setminus F$, where $F=\bigcup\limits_{j=1}^\infty F_j$ and $\capp(F)=0$.
\end{proof}

The idea of using Steklov averages in connection with the renormalized formulation is  developed in \cite{BlPo}. Following this latter paper, we deduce in particular the following result:

\begin{lemma}\label{testu}
If $u$ is a renormalized solution of \eqref{renpb}, then
\begin{multline}\label{Sh}
- \int_Q \Psi(T_k(u))\xi_t \, dxdt - \into \Psi(T_k(u_0))\xi(0, x)\,dx+ \int_Q \adixt{\nabla T_k(u)}\nabla (\psi(T_k(u))\xi) \, dxdt
\\
\leq \int_Q {(\psi(T_k(u)))_h}\xi\, d\mu+\|\psi\|_\infty  \|\xi\|_\infty\|\la_k\| + o(1)_h,
\end{multline} 
and
\begin{multline}\label{S-h}
- \int_Q \Psi(T_k(u))\xi_t  \, dxdt - \into \Psi(T_k(u_0))\xi(0, x)\,dx+ \int_Q \adixt{\nabla T_k(u)}\nabla (\psi(T_k(u))\xi)\, dxdt
\\
\geq \int_Q {(\psi(T_k(u)))_{-h}}\xi\, d\mu - \|\psi\|_\infty  \|\xi\|_\infty\|\la_k\| - o(1)_h,
\end{multline} 
for every nondecreasing $\psi\in W^{1,\infty}(\RR)$ and every nonnegative $\xi \in C^\infty_c([0,T)\times \overline \Omega)$ such that $\psi(0)\xi=0 $ on $(0, T) \times \partial \Omega$,   where   $\Psi(r)=\int_0^r \psi(s)ds$ and $o(1)_h\to 0$ as $h\to 0$. 
\end{lemma}

\begin{proof}
We choose in \eqref{req2} 
\[
v(t,x)=\xi \, (\psi(T_k(u)))_h= \xi\frac1h \int_t^{t+h} \psi(T_k(u))(s, x) \, ds,
\]
where $\xi$, $\psi$ have the properties stated above. Using \cite[Lemma~2.1]{BlPo} we have
\begin{multline*}
\liminf\limits_{h\to 0} \bigg\{ -\int_Q \left(T_k(u)-T_k(u_0)\right) \bigg(\xi\frac1h \int_t^{t+h} \psi(T_k(u))(s, x) \, ds \bigg)_t\,dxdt \bigg\}
\\
\geq - \int_Q \xi_t \int_0^{T_k(u)}\psi(r)dr\,dxdt- \into \xi(0, x) \int_0^{T_k(u_0)}\psi(r) \,drdx.
\end{multline*}
Therefore, we get from \rife{req2}
\begin{multline*} 
- \int_Q \xi_t \int_0^{T_k(u)}\psi(r)dr\,dxdt- \into \xi(0, x) \int_0^{T_k(u_0)}\psi(r) \,drdx \\
+\int_Q \adixt {\nabla T_k(u)}\nabla \left(\xi(\psi(T_k(u)))_h \right)\,dxdt
\\
 \leq \int_Q  \xi \, (\psi(T_k(u)))_h\, 
\,d\mu+\int_Q   \xi (\psi(T_k(u)))_h\,d\la_k  + o(1)_h.
\end{multline*}
In the energy term we can let $h$ go to zero since the averages are continuous in $\psob$. We also   use   that $|(\psi(T_k(u)))_h|\leq   \|\psi\|_\infty$ cap-quasi everywhere in the term with $\la_k$, and we obtain \eqref{Sh}. The proof of \eqref{S-h} is identical using now \cite[Lemma 2.3]{BlPo} for the time derivative. 
\end{proof}

We immediately deduce the following

\begin{corollary}\label{corest} If $u$ is a renormalized solution of \eqref{renpb}, then
\begin{multline*}
\into \int_0^{u(\tau)} \psi(r) \, drdx +\int_0^\tau\into \adixt{\nabla u}{\nabla u}\psi'(u)\,dxdt
\\
\leq \into \int_0^{u_0}  \psi(r)\,drdx + \|\psi\|_\infty \|\mu\|_{\pmisure} \quad \hbox{for a.e. $\tau\in (0,T)$},
\end{multline*}
for every nondecreasing $\psi\in W^{1,\infty}(\RR)$ such that $\psi(0)=0$ and $\psi' $ has compact support. 
\end{corollary}

\begin{proof}  
We use \eqref{Sh}, where we take $\xi=\xi(t)$, with $0\leq \xi\leq 1$.   First we estimate in an obvious way the term with $\mu$.  Then we let $h\to 0$. Finally, since $\psi'$ has compact support, it is also possible to let $k\to \infty$,  using \eqref{limnu}. With  a standard choice of $\xi$ (e.g.\@ a smooth approximation of $\chi_{(0,\tau)}$) we conclude.
\end{proof}

\begin{remark} 
\rm
When $\mu\in L^1(Q)$, one can pass to the limit in \eqref{Sh} and \eqref{S-h} when $h\to 0$ using the continuity of Steklov approximations in $L^1$ and the fact that $\psi$ is bounded. Choosing $\psi'$ with compact support also allows to let $k\to \infty$ and one obtains  
\begin{multline*}
- \int_Q \Psi(u)\xi_t \, dxdt- \into \Psi(u_0)\xi(0, x)\,dx
+ \int_Q \adixt{\nabla u}\nabla (\psi(u)\xi) \, dxdt\\ = \int_Q \psi(u) \xi \mu\,dxdt .
\end{multline*}
Of course one can replace here $\psi$ with $-\psi$ and $\xi$ with $-\xi$, hence the equality will be true for any   Lipschitz function $\psi$ and any $\xi$. We then recover the usual renormalized formulation in case of $L^1$-data. The asymptotic estimate for the energy will  also be proved later  (see \eqref{asy-bm} in Proposition~\ref{stimfett} below), and this proves that this formulation is equivalent to the one given in   \cite{BlMu} for $L^1$-data. 
\end{remark}

We are  now able to show that any renormalized solution satisfies the usual estimates and is, in particular, a distributional solution. To this purpose, we only need to precise  what we mean by $\nabla u$ when $u$ need not belong to any Sobolev space. We follow the definition of generalized gradient introduced in \cite{B6} for functions $u$ whose truncations belong to a  Sobolev space:

\begin{definition}
Let $u:Q\to \mathbb R$ be a measurable function which is almost everywhere finite and such that $T_k(u)\in \psob $ for every $k>0$. Then (see \cite[Lemma~2.1]{B6}) there exists a unique vector-valued function $U$ such that
$$
U= \nabla T_k(u)\chi_{\{|u|<k\}}\quad \hbox{a.e. in  $Q$},\quad \forall  k>0.
$$ 
This function $U$ will be called the gradient of $u$, hereafter denoted by $\nabla u$. When $u\in L^1(0,T;W^{1,1}_0(\Omega))$, it coincides with the usual distributional gradient.
\end{definition}

 We recall the definition of a distributional solution of \eqref{renpb}. Notice that such a  definition  makes sense  for any measure $\mu$, not necessarily diffuse, even if  in our context we are always dealing  with diffuse measures.
 
\begin{definition}
A function $u\in L^1(Q)$ is a distributional solution of problem \eqref{renpb} if  $T_k(u)\in \psob$ for every $k>0$, if    $| \nabla u |^{p-1}  \in L^{1}(Q)$,  and if
\begin{equation}\label{weak}
-\int_Q u \varphi_t \, dxdt +\int_Q a(t,x,\nabla u)  \cdot\nabla \varphi \, dxdt =\int_Q \varphi\ d\mu + \into u_0 \varphi( 0,x) \, dx ,
\end{equation}
for any $\varphi\in C^{\infty}_{c}([0,T)\times\Omega)$. 
\end{definition}

We then have

\begin{proposition}\label{ren-dis}
If $u$ is a renormalized solution of \eqref{renpb}, then for every $k>0$ and $\tau\leq T$,
\be\label{baseR}
\into \Theta_k (u)(\tau) \, dx + \int_{0}^{\tau} \!\! \into |\nabla \tku |^p \, dx dt \leq C k\left( \|\mu\|_{\mathcal{M}(Q)}+ \|u_0\|_{\elle1}\right),
\ee
where $\Theta_k(s)= \int_0^s T_k(t) \,dt$.
Therefore,  
\[
u\in L^\infty(0,T;\elle1)
\]
and
\[
| \nabla u |^{p-1} \in L^{r}(Q) \quad \text{and} \quad a(t,x,\nabla u) \in L^{r}(Q)
\]
for any $r<\frac{N+p'}{N+1}$. Moreover, $u$ is a distributional solution.
\end{proposition}

\begin{proof}
Estimate \eqref{baseR} immediately follows from Corollary~\ref{corest} taking $\psi= T_k$ and using assumption \eqref{coercp}. When $k=1$ one deduces that $ u\in L^\infty(0,T;\elle1)$. Following the results in  \cite{bdgo,BG}, one also deduces the regularity for $|\nabla u|^{p-1}$, hence for $a(t,x,\nabla u)$ in view of \eqref{cont}. In particular, we have $a(t,x,\nabla u)\in L^1(Q)$, and letting $k\to +\infty$ in \eqref{req} we obtain \eqref{weak}, i.e. $u$ is a distributional solution.
\end{proof}


\begin{remark} \rm When $\mu\in \pmisure\cap L^{p'}(0,T;\duale)$ and $u_0\in \elle2$,  there exists a  solution of \eqref{renpb} in the  usual   weak sense (see \cite{l}). Such a solution is also  a renormalized solution. Indeed, in this case $u\in W$ and  satisfies the capacitary estimates of Theorem~\ref{stimcap2}. Proceeding like in the proof of Theorem~\ref{app} one obtains the renormalized formulation for $u$.
\end{remark}

We will now prove that the problem \eqref{renpb} is well-posed in the class of renormalized solutions. Thanks to the robustness of the formulation, the easiest part here is the uniqueness, which comes from the following comparison principle.

\begin{theorem}\label{contrazione}
 Let $u_1$, $u_2$ be two renormalized solutions of problem \eqref{renpb} with data $(u_{01}, \mu_1)$ and $(u_{02}, \mu_2)$ respectively. Then we have
\be\label{contrabis}
\into (u_1-u_2)^+(t)\, dx \leq \| (u_{01}-u_{02})^+\|_{\elle1}+  \| (\mu_1-\mu_2)^+\|_{\pmisure}
\ee
for almost every $t\in (0,T)$. In particular, if  $u_{01}\leq u_{02}$ and $\mu_1\leq \mu_2$ (in the sense of measures), we have $u_1\leq u_2$ a.e. in $Q$. As a consequence, there exists at most one renormalized solution of problem \eqref{renpb}.
\end{theorem}

\begin{proof}
Let   $\la_{k,1}$, $\la_{k,2}$ be the measures given by Definition~\ref{rendef} corresponding to $u_1$, $u_2$. Proposition~\ref{testW} implies
\begin{multline*}
-\int_Q (T_k(u_1)-T_k(u_2))v_t\,dxdt +\int_Q (\adixt {\nabla T_k(u_1)}-\adixt {\nabla T_k(u_2)})\nabla v\,dxdt
\\
= \int_Q  v\,d(\mu_1-\mu_2)+\int_Q  v\,d\la_{k,1}-\int_Q  v\,d\la_{k,2} +\into (T_k(u_{01})-T_k(u_{02}))v(0, x)\,dx 
\end{multline*}
for every $v\in W\cap\parelle\infty$ such that $v(T)=0$. Consider the function 
$$
w_h(t,x)=\frac1h\int_t^{t+h} \frac1\vep T_\vep\left( T_k(u_1)-T_k(u_2)\right)^+(s, x)\,ds,
$$
Given $\xi\in C_c^\infty([0,T))$, $\xi\geq 0$, take $v=w_h\xi$ as test function. Observe that both $w_h$ and $(w_h)_t$ belong to $L^p(0,T;V)\cap\parelle\infty$  for \(h > 0\) sufficiently small, hence $w_h\in W\cap\parelle\infty$.  Moreover we have
$$
w_h\to \frac1\vep T_\vep\left( T_k(u_1)-T_k(u_2)\right)^+\qquad \hbox{strongly in $\psob$.}
$$
Using that $0\leq w_h\leq 1$ almost everywhere, hence $0\leq  w_h\leq 1$ cap-quasi everywhere, we have
\begin{multline}\label{ineh}
-\int_Q \left[ (T_k(u_1)-T_k(u_2))-(T_k(u_{01})-T_k(u_{02}))\right] (w_h\xi)_t\,dxdt 
\\
+\int_Q (\adixt {\nabla T_k(u_1)}-\adixt {\nabla T_k(u_2)})\nabla w_h\xi\,dxdt
\\
\leq \|\xi\|_\infty \left( \| (\mu_1-\mu_2)^+\|_{\pmisure} + \|\la_{k,1}\|_{\pmisure} + \|\la_{k,2} \|_{\pmisure}\right).
\end{multline}
Using the monotonicity of $T_\vep(s)$ we have (see \cite[Lemma 2.1]{BlPo})
\begin{multline*}
\liminf\limits_{h\to 0} \left\{ -\int_Q \left[ (T_k(u_1)-T_k(u_2))-(T_k(u_{01})-T_k(u_{02}))\right] (w_h\,\xi)_t\,dxdt \right\}
\\
\geq -\int_Q \hat\Theta_\vep(T_k(u_1)-T_k(u_2))\xi_t\,dxdt- 
\into \hat\Theta_\vep(T_k(u_{01})-T_k(u_{02}))\xi(0)\,dx
\end{multline*}
where $\hat\Theta_\vep(s)=\int_0^s \frac1\vep T_\vep(r)^+dr$.  Therefore, letting $h\to0$ in \eqref{ineh} we obtain
\begin{multline*}
-\int_Q \hat\Theta_\vep(T_k(u_1)-T_k(u_2))\xi_t\,dxdt\\
+
\frac1\vep\int_Q (\adixt {\nabla T_k(u_1)}-\adixt {\nabla T_k(u_2)})\nabla T_\vep(T_k(u_1)-T_k(u_2))^+\xi\,dxdt
\\
\leq \into \hat\Theta_\vep(T_k(u_{01})-T_k(u_{02}))\xi(0 )\,dx\\
\m
+ \|\xi\|_\infty \left( \| (\mu_1-\mu_2)^+\|_{\pmisure} + \|\la_{k,1}\|_{\pmisure} + \|\la_{k,2} \|_{\pmisure}\right).
\end{multline*}
Using \eqref{monot} and letting $\vep\to 0$ we deduce
\begin{multline*}
-\int_Q (T_k(u_1)-T_k(u_2))^+\xi_t\,dxdt
\leq \into (T_k(u_{01})-T_k(u_{02}))^+\xi(0)\,dx\\
+ \|\xi\|_\infty \left( \| (\mu_1-\mu_2)^+\|_{\pmisure} + \|\la_{k,1}\|_{\pmisure} + \|\la_{k,2} \|_{\pmisure}\right)
\end{multline*}
and letting 
$k\to \infty$ we obtain, thanks to \eqref{limnu},
$$
-\int_Q (u_1-u_2)^+\xi_t\,dxdt\leq \|\xi\|_\infty\left( \|(u_{01}-u_{02})^+\|_{\elle1}+  \| (\mu_1-\mu_2)^+\|_{\pmisure} \right)
$$
for every nonnegative $\xi\in C_c^\infty[0,T)$.  Of course the same inequality holds for any $\xi\in W^{1,\infty}(0,T)$ with compact support in \([0, T)\). Take then $\xi(t) =1- \frac1\vep\,T_\vep(t-\tau)^+$, where $\tau\in (0,T)$;   since $u_1$, $u_2\in \limitate 1$, by letting $\vep \to 0$ we have
$$
-\int_Q (u_1-u_2)^+\xi_t\,dxdt= \frac1\vep\int_{\tau}^{\tau+\vep}\into (u_1-u_2)^+\,dxdt\to \into (u_1-u_2)^+(\tau)\,dx
$$
for almost every $\tau\in (0,T)$. Using in the right hand side that $\|\xi\|_\infty\leq 1$ we get  \eqref{contrabis}.
\end{proof}

\begin{remark}\rm
The above result continue to hold if we replace  condition \eqref{monot} with the more general
$$
[a(t,x,\xi) - a(t,x,\eta)] \cdot (\xi - \eta) \geq  0\qquad \forall \xi,\eta\in \rn. 
$$
\end{remark}

\begin{remark}\label{comp-anom}\rm 
If $\mu\in \pmisure\cap L^{p'}(0,T;\duale)$, then $\mu^+$ is a diffuse measure (since $\mu$ is diffuse) but we can not infer that $\mu^+$ belongs to $L^{p'}(0,T;\duale)$. However, if $u$ is a standard weak solution corresponding to $\mu$, and $v$ is a renormalized solution corresponding to $\mu^+$, one can deduce that $u\leq v$. It is enough to proceed as in the above proof  using  
\[
(t, x) \longmapsto \xi(t)\,\frac1h\int_t^{t+h} \frac1\vep T_\vep\left( u-T_k(v)\right)^+(s, x)\,ds
\]
as test function.
\end{remark}

We immediately deduce the $L^1$-contraction estimate.

\begin{corollary}\label{compari}
 Let $u_1$, $u_2$ be two renormalized solutions of problem \eqref{renpb} with data $(u_{01}, \mu_1)$ and $(u_{02}, \mu_2)$ respectively. Then we have
\be\label{contra}
\into |u_1-u_2|(t)\, dx \leq \|u_{01}-u_{02}\|_{\elle1}+  \| \mu_1-\mu_2\|_{\pmisure}
\ee
for almost every $t\in (0,T)$. 
\end{corollary}

\vskip1em
As usual for nonlinear equations with measure data, we  will prove the existence of solutions  through  approximation of the data $\mu$, $u_0$  with smooth functions. We will need the following proposition which collects some known results in the literature.

\begin{proposition}\label{pro}
Let $(u_n)$ be a sequence of solutions of problem 
\begin{equation}\label{senzah}
\begin{cases}
    (u_n)_{t}-\dive(a(t,x,\nabla  u_n)) =\mu_n & \text{in}\ Q,\\
    u_n=u_{0n}  & \text{on}\ \{0\} \times \Omega,\\
 u_n=0 &\text{on}\ (0,T)\times\partial\Omega,
  \end{cases}
\end{equation} 
  where  $(u_{0n})$ strongly converges to $u_0$ in $L^1(Q)$ and $(\mu_n)$ is a bounded sequence in \(L^1(\Omega)\). Then,
\begin{equation}\label{pro1}
\|u_n\|_{L^\infty (0,T; L^1 (\Omega))}\leq C,
\end{equation}
and
\begin{equation}\label{pro2}
\int_Q|\nabla T_k (u_n)|^p\ dxdt\leq Ck \qquad \forall k>0.
\end{equation}
Moreover, there exists a measurable function $u$ such that $T_k (u)\in\psob$ for any $k>0$, $u\in L^\infty (0,T; L^1 (\Omega))$, and, up to  a subsequence,  we have
\be\label{conver}
\begin{split}
& u_n \rightarrow u \qquad \hbox{a.e. in $Q$    and strongly in $L^1 (Q)$,}\\
&T_k (u_n)\rightharpoonup T_k (u)\qquad \hbox{weakly in $ \psob$  and a.e. in $Q$,}\\
&\nabla u_n\rightarrow \nabla u \qquad \hbox{a.e. in $Q$.}\\
&|\nabla u_n|^{p-2}\nabla u_n\rightarrow |\nabla u|^{p-2}\nabla u\qquad  \hbox{ in $\luq$.}
\end{split}
\ee
\end{proposition}

\begin{proof}  
Estimates \eqref{pro1}-\eqref{pro2} are classical since the work \cite{BG}.  The a.e. convergence of $(u_n)$ can be proved by using properties of truncations and the Marcinkiewicz estimates, see e.g. \cite{BlMu} or \cite{po}. The a.e. convergence of $(\nabla u_n)$ is proved in \cite{bdgo} when $u_{0n}=0$; however, their proof extends to the case of sequences $(u_{0n})$ converging in $L^1$, the only difference being to include the initial condition  $u_{0n}$ in the so-called Landes approximation. This latter technical point can be found e.g. in \cite{po}.
\end{proof}

In order to prove the existence of a renormalized solution, we 
consider the convolution
$$
\mu_n=  \rho_n \ast \mu,
$$
where $\rho_n$ is a sequence of standard mollifiers. In this way, we can use the equi-diffusion property (see Proposition~\ref{equi}) which plays  a crucial role in our approach.
We will  state the convergence result in  a slightly more general form which can be applied in case of lower order terms as well.

\begin{proposition}\label{jn-app} 
Let $(u_{0n})$ be a sequence converging strongly to $u_0$ in $\elle1$, let $(f_{n})$ be a sequence of bounded functions converging to $f$ strongly in $\parelle1$, and let $\mu_{n}= \mu\ast \rho_n$. Let $\ujn$ be the solution of 
\begin{equation}\label{appjn}
\begin{cases}
    (\ujn)_t-\dive(\adixt {\nabla \ujn} )=f_{n}+ \mu_n & \text{in}\ Q,\\
    \ujn =u_{0n}  & \text{on}\ \{0\} \times \Omega,\\
 \ujn=0 &\text{on}\ (0,T)\times\partial\Omega.
  \end{cases}
\end{equation} 
Then there exists  a function $u$ such that  \eqref{conver} holds true up to subsequences and  
 $u$ is a renormalized solution, in the sense of Definition~\ref{rendef}, corresponding to $f+\mu$.
\end{proposition}

\begin{proof} We split the proof in four steps:

\vskip0.3em
\noindent
\emph{Step 1.} Basic estimates.
\smallskip

Both sequences $(\mu_n)$ and $(f_n)$ are bounded in \(L^1(Q)\).
Therefore the solutions $\ujn$ satisfy the estimates  recalled in Proposition~\ref{pro} and up to a subsequence \eqref{conver} holds. Moreover, Theorem~\ref{stimcap} applies and gives the estimate
\be\label{cap-unif}
\capp\{ |\ujn|>k\} \leq C \max{\big\{ k^{-\frac1p},k^{-\frac1{p'}} \big\}} \quad \forall k \ge 1,\quad \forall n \ge 1.
\ee
Finally, for $\de>0$ consider the function $h_\de(s)=\frac1\de (T_{k+\de}(s)-T_k(s))$, which is a piecewise linear  odd function vanishing for $|s|\leq k$ and constant for $|s|>k+\de$.  Let $\xi_\eps(t) =1- \frac1\vep\,T_\vep(t- T + 2\eps)^+$, where $\eps > 0$; this function belongs to \(W^{1, \infty}(0, T)\) and has compact support in \([0, T)\). By using $h_\de(\ujn) \xi_\eps$ as test function in \eqref{appjn} and letting \(\eps \to 0\), we have
\begin{multline*}
\frac1\de\int\limits_{\{k<|\ujn|<k+\de\}}\adixt{\nabla \ujn}\nabla \ujn\,dxdt
\\
\leq \int_Q h_\de(\ujn)\mu_n\,dxdt
+ \int_Q f_{n} h_\de(\ujn)\,dxdt+ \int\limits_{\{|u_{0n}|>k\}} |u_{0n}|\,dx,
\end{multline*}
which implies, in particular,
\begin{multline}
\label{codajn}
\frac1\de\int\limits_{\{k<|\ujn|<k+\de\}}\adixt{\nabla \ujn}\nabla \ujn \, dxdt
\\
\begin{aligned}
& \leq \int\limits_{\{|\ujn|>k\}} |\mu_n|\,dxdt
+ \int\limits_{\{|\ujn|>k\}}|f_{n}|\,dxdt+  \int\limits_{\{|u_{0n}|>k\}} |u_{0n}|\,dx.
\end{aligned}
\end{multline}

\vskip0.3em
\noindent
\emph{Step 2.} Equation satisfied by truncations.
\smallskip

For $\delta>0$ small, consider
the functions $S_{k,\delta}$ and $T_{k,\delta}$ given by \eqref{skd} and \eqref{tkd}, respectively.
Recall that $T_{k,\delta}$ converges pointwise to $T_{k}$ as $\delta \to 0$.\\
Given $\varphi \in C_c^\infty(Q)$, multiply the equation solved by $\ujn$ by $\skdujn \varphi$.  We then have
\begin{multline*}\label{equatjn}
\tkdujn_t  -\dive\left(  \skdujn\adixt{\nabla \ujn}\right)  
\\ = f_{n}+ \mu_n + (\skdujn-1)(f_{n}+\mu_n) \\+\frac{1}{\delta}\adixt{\nabla \ujn}\nabla \ujn {\rm sign}(u_n)\chi_{\{k< |u_n|<k+\delta\}} \quad \text{in}\ \mathcal{D}'(Q).
\end{multline*}
 Let
\[\label{defnu}
\Lambda^k_{n,\de}:= (\skdujn-1)(f_{n}+\mu_n) +\frac{1}{\delta}\adixt{\nabla \ujn}\nabla \ujn  {\rm sign}(u_{n})\chi_{\{k< |\ujn|<k+\delta\}}
\]
With this notation we have
\begin{equation}
\label{equatjn-bis}
(T_k(\ujn))_t-\dive\left(  \skdujn\adixt{\nabla \ujn}\right) = f_{n}+ \mu_n + \Lambda^k_{n, \delta} \quad \text{in}\ \mathcal{D}'(Q).
\end{equation}
Thanks to estimate \eqref{codajn}, and since $|S_{k,\de}|\leq 1$,  the functions \(\Lambda^k_{n,\de}\) are bounded in $L^{1}(Q)$ uniformly with respect to $\de$ and $n$.  Indeed, to be more precise, \eqref{codajn} and the definition of $S_{k,\de}$ imply the estimate 
\be\label{stimanu}
\begin{split}
\|\Lambda^k_{n,\de}\|_{\parelle1} 
& \leq 
\int\limits_{\{ |\ujn|>k\}} |f_{n}+ \mu_n|\,dxdt+ \frac1\de\int\limits_{\{k<|\ujn|<k+\de\}}\adixt{\nabla \ujn}\nabla \ujn\,dxdt\\
& \leq 
2\int\limits_{\{|\ujn|>k\}} |\mu_n|\,dxdt
+2 \int\limits_{\{|\ujn|>k\}}|f_{n}|\,dxdt+  \int\limits_{\{|u_{0n}|>k\}} |u_{0n}|\,dx\,.
\end{split}
\ee

\vskip0.3em
\noindent
\emph{Step 3.} Limit as $n$ goes to infinity.

\smallskip
Applying Proposition~\ref{pro} to the sequence $(\ujn)$,  there exists a function $u\in L^1(Q)$ such that $T_k(u)\in \psob$ for all $k>0$ and, up to a subsequence, 
$$
\ujn\to u,\quad \nabla \ujn\to \nabla u \quad \hbox{a.e. in $Q$.}
$$
Let now \((\delta_n)\) be a sequence of positive numbers converging to \(0\), as \(n \to \infty\).  Using the definition of $S_{k,\de}(t)$ (see \rife{skd}), the fact that $\nabla u=0$ a.e. in  $\{|u|=k\}$ and $\adixt 0=0$ (a consequence of \rife{coercp}), we deduce that
$$
S_{k, \delta_n}(u_n)\adixt{\nabla \ujn} \to  \adixt{\nabla T_k(u)}\quad \hbox{a.e. in $Q$.}
$$
Since $S_{k, \delta_n}(u_n)\adixt{\nabla \ujn}$ is bounded in $L^{p'}(Q)$, it also follows that
$$
S_{k, \delta_n}(u_n)\adixt{\nabla \ujn} \rightharpoonup \adixt{\nabla T_k(u)}\quad \hbox{weakly in $\pelle{p'}$.}
$$
By properties of convolution, clearly we have $\mu_n \overset{*}{\rightharpoonup} \mu$ in the weak$^*$ topology of measures.  Finally, since the sequence $(\Lambda^k_{n, \delta_n})_n$ is bounded in \(L^1(\Omega)\), 
there exists a bounded measure $\la_k$ such that, up to  a subsequence, \((\Lambda^k_{n, \delta_n})_n\) converges to $\la_k$ in the weak$^*$ topology of $C(\overline Q)'$.  
We deduce from \eqref{equatjn-bis} that $u$ satisfies 
\begin{multline*}
-\int_Q T_k(u)\vfi_t\,dxdt +\int_Q \adixt {\nabla T_k(u)}\nabla \vfi\,dxdt
\\
= \int_Q \vfi f\,dxdt+ 
\int_Q \vfi\,d\mu+\int_Q \vfi\,d\la_k+\into T_k(u_{0})\vfi(0, x)\,dx,
\end{multline*}
for every $\vfi\in C^\infty_c([0,T)\times \Omega)$.

\vskip0.3em
\noindent
\emph{Step 4.} Estimate on ${\la_k}$.

\smallskip
By weak$^*$ lower semicontinuity of the norm, we have
\[\label{wls}
\|\la_k\|_{\pmisure}\leq \liminf_{n\to \infty} \|\Lambda^k_{n, \delta_n}\|_{\pmisure}.
\]
Since, by Proposition~\ref{equi},  the sequence $(\mu_n)$ is equidiffuse,  thanks to the uniform estimate \eqref{cap-unif} we deduce that
$$
\lim\limits_{k\to \infty} \sup\limits_{n} \int\limits_{\{|\ujn|>k\}} |\mu_n|\,dxdt=0.
$$
By the equi-integrability of $(u_{0n})$ and $(f_{n})$, the last two terms in \eqref{stimanu} satisfy a similar property. Then, by \eqref{stimanu} we have
$$
\|\la_k\|_{\pmisure}\leq \liminf\limits_{n\to \infty} \|\Lambda^k_{n, \delta_n}\|_{\pmisure}  \leq  \vep_k,
$$
where $\vep_k$ is some quantity which tends to zero as $k\to\infty$. 
Letting   $k\to \infty$ we conclude that $\la_k$ satisfies \eqref{limnu}, hence $u$ is a renormalized solution.
\end{proof}

\vskip1em
We immediately deduce by the  previous results the following 

\begin{theorem}Let $\mu$ be a diffuse measure, and $u_0\in \elle1$. Assume that \eqref{coercp}-\eqref{monot} hold true. Then there exists  a unique renormalized solution of \eqref{renpb} in the sense of Definition~\ref{rendef}.
\end{theorem}
\vskip1em

Let us point out that the same arguments used in the construction of a renormalized solution could also be employed for  a stability result of renormalized solutions corresponding to  data $\mu$ which are weakly 
converging and equidiffuse. 
 However, to this purpose one would need the capacitary estimate of Theorem \ref{stimcap} to hold  for renormalized solutions as well. A major problem here is that solutions with measure data may not possess a cap-quasicontinuous representative, hence  $\capp(\{|u| > k\})$ is not well defined for all renormalized solutions $u$. To overcome this obstacle, the following weak form of  Theorem~\ref{stimcap} can be proved.

\begin{proposition}\label{tm} Let $u$ be a renormalized solution of \eqref{renpb}. Then, for every $m \ge 1$ there exist positive functions $z_m$, $w_m
\in W$ such that:
\begin{itemize}
\item[(i)]  $-w_m\leq T_m(u)\leq z_m$ a.e. in $Q$.
\vskip0.5em
\item[(ii)] $  \|z_m\|_W+ \|w_m\|_W \leq C \max\big\{ m^{\frac1p},m^{\frac1{p'}}\big\}$
\end{itemize}
where $C=C(\|\mu\|_{\cM(Q)},\|u_0\|_{\elle1}, p, \alpha,\beta,\|b\|_{\pelle{p'}})$.
In particular,
 $$
 \capp(\{z_m>m\})+  \capp(\{w_m>m\})\leq C\max\left\{ \frac1{m^{\frac1p}},\frac1{m^{\frac1{p'}}}\right\}.
 $$ 
\end{proposition}
 
\begin{proof}
The proof is done following that of Theorem~\ref{stimcap}. Let us only precise a  few technical modifications. First of all, we start by  considering the case where $\mu\geq 0$  and $u_0 \ge 0$,  which implies that $u\geq 0$ by Theorem~\ref{contrazione}.  
The estimate in Step 1 of Theorem~\ref{stimcap} follows from Proposition~\ref{ren-dis}. Next we construct  
the function $z_m$ in the same way as in Step 2 of Theorem~\ref{stimcap} (replacing  $-\Delta_p(\cdot)$ with $-\dive(a(t,x,\nabla (\cdot) ))$): here we assume that $T_m(u(T))$ is well defined   (there is however no loss of generality, otherwise one can use some sequence $t_n\uparrow T$ such that $T_m(u)(t_n)$ is well defined, recall that $u\in L^\infty(0;T;\elle1)$ and such a sequence of Lebesgue points certainly exists). Using assumptions \eqref{coercp} and \eqref{cont}, one obtains the estimate on  $z_m$:
$$
\|z_m\|_W \leq C \max\{ m^{\frac1p},m^{\frac1{p'}}\}\qquad \hbox{where $C=C(\|\mu\|, \|u_0\|,p,\alpha,\beta, \|b\|_{\pelle{p'}})$.}
$$
Finally, to conclude that $z_m\geq T_m(u)$, we show that
\be\label{postr}
T_m(u)_t- \dive(a(t,x,\nabla T_m(u)))\geq 0.
\ee
To this purpose, we use  \eqref{Sh} with $\psi= - S_{m,\de}(r^+)$,  where $S_{m,\de}$ is defined in \eqref{skd}. Since $\mu\geq 0$,  we have, for every nonnegative $\xi\in C^\infty_c([0,T)\times \Omega)$,
$$
\int_Q \xi {\psi(T_k(u))_h}\,d\mu+\|\xi\|_\infty\|\psi\|_\infty\|\la_k\| \leq  \|\xi\|_\infty \|\la_k\| \to 0 .
$$
Then we let $h\to 0$, $k\to \infty$ and we obtain
\begin{multline*}
- \int_Q \xi_t \int_0^{u} S_{m,\de}(r^+) \,dr \, dxdt- \into \xi(0, x)\, \int_0^{u_0} S_{m,\de}(r^+) \, drdx
\\
+
\int_Q\adixt{\nabla u}\nabla (S_{m,\de}(u^+)\xi) \, dxdt\geq 0.
\end{multline*}
Recalling that $u\geq 0$, dropping the term with $S'_{m,\de}$ which is negative, and letting $\de\to0$ we deduce \eqref{postr}. Therefore we complete Step 3 of Theorem~\ref{stimcap} and the conclusion in case that $\mu\geq0$ and $u_0\geq 0$. Since $-\mu^-\leq \mu\leq \mu^+$ and \(-u_0^- \le u_0 \le u_0^+\), the general case can  be obtained  comparing $u$ with the solutions corresponding to  $\mu^+$, \(u_0^+\) and $-\mu^-$, \(-u_0^-\), which is possible  thanks to the comparison principle for renormalized solutions (Theorem~\ref{contrazione}). 
\end{proof}

Let us notice that, in the previous proof, we do not need anymore the regularizing procedure of Step 5 of Theorem~\ref{stimcap} since we can compare directly renormalized solutions which are defined even if data do not belong to $L^{p'}(0,T;\duale)$.

We deduce now an energy estimate for the  renormalized solutions which extends  the classical one known  in case of $L^1$ data.

\begin{proposition}\label{stimfett} Let $u$ be a renormalized solution of \eqref{renpb}. Then, for every $ m \geq 1$ there exists a  cap-quasi open set $E_m \subset Q$ such that 
\be\label{capem}
\capp(E_m)\leq C\max\left\{\frac{1}{m^{\frac{1}{p}}},\frac{1}{m^\frac{1}{p'}}\right\} 
\end{equation}
where $C = C( \|\mu\|_{\mathcal{M}(Q)}, \|u_0\|_{\elle1},  p, \al,\beta,\|b\|_{\pelle{p'}})$, and, for every $\de\in(0,1]$, 
\be\label{fett}
\frac1\de\int\limits_{\{m<|u|<m+\de\}} \adixt{\nabla u}\nabla u\,dxdt\leq \int\limits_{\{|u_0|>m\}} |u_0|\,dx + |\mu|\left(E_m\right).
\ee
In particular, we have
\be\label{asy-bm}
\lim\limits_{m\to \infty} \int\limits_{\{m<|u|<m+1\}}\!\!\!\adixt{\nabla u}\nabla u\,dxdt=0.
\ee
Moreover, we have the estimate
\be\label{estnum}
\|\la_m\|_{\cM(Q)} \leq  \int\limits_{\{|u_0|>m\}} |u_0|\,dx+ 2\,  |\mu|(E_m) 
\ee
where $\la_m$ is the measure associated to the renormalized equation of $u$.
\end{proposition}

\begin{proof}  Let $\psi_{m,\de}(s)$ be a $C^2(\RR)$ function which is  nondecreasing, odd and  such that $\psi_{m,\de}(s)=0$ if $|s|\leq m$ and $\psi_{m,\de}(s)=1$ if $s\geq m+\de$.  We use \eqref{Sh} with $\xi\in C^\infty_c([0,T))$ and $\psi=\psi_{m,\de}$.  Assuming $k>m+1$, we obtain
\begin{multline*}
  - \int_Q \xi_t \int_0^{T_k(u)} \psi_{m,\de}(r) \,dr\,dxdt + \int\limits_{\{m<|u|<m+\de\}}\!\!\!\!\! \adixt{\nabla u}\nabla u\,\psi_{m,\de}'(u)\xi \,dxdt
  \\
\leq    \int_Q\xi  {\big(\psi_{m,\de}(T_k(u))\big)_h}\,d\mu+ \|\xi\|_\infty\|\la_k\|_{\cM(Q)} + \|\xi\|_\infty \int\limits_{\{|u_0|>m\}}\!\!\! |u_0| \,dx + o(1)_h
\end{multline*}
where $o(1)_h$ tends to zero as $h\to 0$. Taking $\xi_\vep\in C^\infty_c([0,T))$ such that $(\xi_\vep)_t\leq 0$  and $\xi_\vep\nearrow 1$ pointwise in $(0,T)$   we get
\begin{multline}\label{fettm}
\int\limits_{\{m<|u|<m+\de\}} \!\!\!\!\!\adixt{\nabla u}\nabla u \,\psi_{m,\de}'(u)\,dxdt 
\\ \leq   \int_Q | {\psi_{m,\de}(T_k(u))_h}|\,d|\mu|  +\|\la_k\|_{\cM(Q)}+ \int\limits_{\{|u_0|>m\}} \!\!\! |u_0| \,dx 
 +  o(1)_h.
\end{multline}
Since $k>m+1$, we have $\psi_{m,\de}(T_k(u))= \psi_{m,\de}(T_{m+1}(u))$, hence using Proposition~\ref{tm}  we have 
\be\label{stimm}
\psi_{m,\de}(T_k(u))^+ \leq \psi_{m,\de}(z_{m+1}) \quad \text{and} \quad  \psi_{m,\de}(T_k(u))^- \leq  \psi_{m,\de}(w_{m+1}),
\ee
and since the inequalities are preserved taking the average and the cap-quasi continuous representatives we deduce from \eqref{fettm}:
\begin{multline*}
\int\limits_{\{m<|u|<m+\de\}} \!\!\!\!\!\adixt{\nabla u}\nabla u \,\psi_{m,\de}'(u)\,dxdt\\
\leq \int_Q \left[  { (\psi_{m,\de}(z_{m+1}))_h}+  {(\psi_{m,\de}(w_{m+1}))_h} \right]\,d|\mu| +\|\la_k\|_{\cM(Q)}\\
+  \int\limits_{\{|u_0|>m\}} \!\!\! |u_0| \,dx 
+o(1)_h.
\end{multline*}
Applying   Lemma~\ref{tech1}, for any $z\in W$ we have   $ \psi_{m,\de}(z)\in S\cap L^\infty(Q)$; hence, it follows from  Lemma~\ref{cap-stek}
that, up to subsequences, ${  {(\psi_{m,\de}(z))_h}}\to  \psi_{m,\de}(  z)$ cap-quasieverywhere. We deduce, by  dominated convergence, 
that  $ (\psi_{m,\de}(z_{m+1}))_h$ and $ { (\psi_{m,\de}(w_{m+1}))_h}$ both converge in $L^1(d|\mu|)$. Then, passing to the limit 
as $h\to 0$ and then $k\to \infty$ we obtain
\begin{multline*}
\int\limits_{\{m<|u|<m+\de\}}\!\!\!\!\!\! \adixt{\nabla u}\nabla u\,\psi_{m,\de}'(u) \,dxdt   \\
\le \int_Q \left[  { (\psi_{m,\de}(z_{m+1}))}+  {(\psi_{m,\de}(w_{m+1}))} \right]d|\mu| + \int\limits_{\{|u_0|>m\}} \!\!\! |u_0| \,dx
.
\end{multline*}
Setting $E_m= \{(t,x): z_{m+1}> m\} \cup\{(t,x):  w_{m+1}>m\}$, we have, by definition of $\psi_{m,\de}$, that
$$
\int_Q \left[  { (\psi_{m,\de}(z_{m+1}))}+  {(\psi_{m,\de}(w_{m+1}))} \right]d|\mu| \leq |\mu|\left( E_m\right),
$$
hence we get 
\be\label{2inf}
\int\limits_{\{m<|u|<m+\de\}}\!\!\!\!\!\! \adixt{\nabla u}\nabla u\,\psi_{m,\de}'(u) \,dxdt \leq    |\mu|\left( E_m\right) + \int\limits_{\{|u_0|>m\}} \!\!\! |u_0| \,dx,
\ee
for every nondecreasing, odd function $\psi\in C^2(\RR)$ such that $\psi_{m,\de}=0$ if $|s|\leq m$ and $\psi_{m,\de}=1$ if $s\geq m+\de$. By approximations with $C^2$ functions, the same inequality will be satisfied by the piecewise linear function $\psi_{m,\de}=\frac1\de T_{\de}(s-T_m(s))$, hence we get \eqref{fett}.  Since
$$
\capp(E_m)\leq \frac{\|z_{m+1}\|_W}{m}+ \frac{\|w_{m+1}\|_W}{m} \,,
$$
the estimate \eqref{capem} follows from Proposition~\ref{tm}. Then, since $\mu$ is diffuse, \rife{fett} implies  \eqref{asy-bm}.

Finally, we want to  estimate $\|\la_m\|_{\cM(Q)}$. To this aim, 
we use  Lemma~\ref{testu} and in particular \eqref{Sh}
with $\psi= \psi_{m,\de}(-s^-)$ and $\eqref{S-h}$ with $\psi= \psi_{m,\de}(s^+)$,  again with $\psi_{m,\de} \in C^2(\RR)$ (and with the same properties as above). We subtract the two inequalities and  we find
\begin{multline*}
-\int_Q \xi_t \int_{T_k(u_0)}^{T_k(u)} (\psi_{m,\de}(-r^-)- \psi_{m,\de}(r^+))(r) \, dr\\
 +
\int_Q \adixt{\nabla T_k(u)}\nabla [\xi(\psi_{m,\de}(-T_k(u)^-)- \psi_{m,\de}(T_k(u)^+))] \, dxdt
\\
\leq \int_Q \xi  {\psi_{m,\de}(-T_k(u)^-)_h} \, d\mu - \int_Q \xi  {\psi_{m,\de}(T_k(u)^+)_{-h}} \, d\mu + 
2 \|\xi\|_\infty \|\la_k\|_{\cM(Q)}+o(1)_h.
\end{multline*}
We add  \eqref{req} (with $\vfi=\xi$) to the above inequality, and we  set $\zeta_{m,\de}(s)= \psi_{m,\de}(-s^-)- \psi_{m,\de}(s^+) +1$. Note that $\zeta_{m,\de}(s)= - {\rm sign}(s) \psi_{m,\de}(s)+1$ and that $\zeta_{m,\de}\to \chi_{|s|\leq m}$ as $\de\to 0$.  We obtain
\begin{multline*}
-\int_Q \xi_t \int_{T_k(u_0)}^{T_k(u)} \zeta_{m,\de}(r)dr\\ +
\int_Q \adixt{\nabla T_k(u)}\nabla [\xi(\zeta_{m,\de}(T_k(u))]dxdt
- \int_Q \xi d\mu- \int_Q \xi\,d\la_k
\\
\leq \int_Q \xi  {\psi_{m,\de}(-T_k(u)^-)_h}d\mu - \int_Q \xi  {\psi_{m,\de}(T_k(u)^+)_{-h}}d\mu + 
2 \|\xi\|_\infty \|\la_k\|_{\cM(Q)} +o(1)_h.
\end{multline*}
Using \eqref{stimm} and since $\zeta_{m,\de}(T_k(u))= \zeta_{m,\de}(u)$ for $ k>m+1$ we deduce
\begin{multline*}
-\int_Q \xi_t \int_{u_0}^{u} \zeta_{m,\de}(r) \, dr
 +
\int_Q \adixt{\nabla  u}\nabla [\xi(\zeta_{m,\de}(u)] \,dxdt - \int_Q \xi \, d\mu 
\\
\le \|\xi\|_\infty \int_Q \left[  { (\psi_{m,\de}(z_{m+1}))_{-h}}+  {(\psi_{m,\de}(w_{m+1}))_h} \right]\,d|\mu| + 
3 \|\xi\|_\infty \|\la_k\|_{\cM(Q)}+o(1)_h,
\end{multline*}
The right hand side can be estimated as   before, letting $h\to 0$,  obtaining
\begin{multline*}
-\int_Q \xi_t \int_{u_0}^{u} \zeta_{m,\de}(r) \, dr
 +
\int_Q \adixt{\nabla u}\nabla \xi \zeta_{m,\de}(u)\,dxdt - \int_Q \xi \, d\mu 
\\
\leq   - \int_Q \adixt{\nabla u}\nabla u\,\zeta_{m,\de}'(u) \xi\, dxdt + \|\xi\|_\infty |\mu|(E_m) + 
3 \|\xi\|_\infty \|\la_k\|_{\cM(Q)}
\end{multline*}
Now let $k\to \infty$ using \eqref{limnu}. Since $\zeta_{m,\de}' = - {\rm sign}(s) \psi_{m,\de}'(s)$, thanks to \eqref{2inf} we end up with  
\begin{multline*}
-\int_Q \xi_t \int_{u_0}^{u} \zeta_{m,\de}(r) \, dr
 +
\int_Q \adixt{\nabla u}\nabla \xi \zeta_{m,\de}(u)\,dxdt - \int_Q \xi \, d\mu 
\\
\leq   \|\xi\|_\infty  \int\limits_{\{|u_0|>m\}}|u_0|\,dx +2 \|\xi\|_\infty |\mu|(E_m) .
\end{multline*}
Letting $\de\to 0$ we find in the left hand side the term $T_m(u)_t-\dive(\adixt{\nabla T_m(u)})-\mu$ which, using \eqref{req},  coincides with $\la_m$, hence
\[
\int_Q \xi d\la_m 
\leq   \|\xi\|_\infty  \int\limits_{\{|u_0|>m\}}|u_0|\,dx +2 \|\xi\|_\infty |\mu|(E_m).
\]
The same inequality can  be obtained  for $-\la_m$  using now \eqref{Sh} with $\psi= \psi_{m,\de}(s^+)$ and $\eqref{S-h}$ with $\psi= \psi_{m,\de}(-s^-)$. Finally we conclude  
$$
\left | \int_Q \xi d\la_m \right|
\leq   \|\xi\|_\infty  \bigg[ \int\limits_{\{|u_0|>m\}}|u_0|\,dx +2 |\mu|(E_m)\bigg]
$$
which implies \eqref{estnum}.
\end{proof}

\begin{remark}
\rm Thanks to Proposition~\ref{stimfett}, it is possible to prove the stability of renormalized solutions with respect to a sequence of data  $(\mu_n)$ which are weakly converging (in the sense of measures) and equidiffuse. In particular,   the estimate \eqref{estnum} implies that the condition \eqref{limnu} on the sequence $(\la_k)$   holds  uniformly (hence it is stable)   when the measures $\mu_n$ are   equidiffuse.  
\end{remark}

Finally, we conclude this section by showing   that Definition~\ref{rendef} implies that $u$ is a renormalized solution in the sense of \cite{dpp}.  Since both solutions have been proved to be unique, in particular this proves that the formulations are actually equivalent. 

\begin{theorem}
Let $u$ be a renormalized solution according to Definition~\ref{rendef}, and let $\mu$ be  split as in \eqref{cap3c1}, namely 
$$
\mu=f-\dive(G)+g_t,\quad \hbox{$f\in \elle1$, $G\in \parelle{p'}$, $g\in L^p(0,T;V)$.}
$$
Then $u$ satisfies:
\be\label{regug}
u-g\in \limitate1,\qquad T_k(u-g)\in \psob\quad \forall k>0,
\ee
\be\label{asyug}
\lim\limits_{h\to \infty} \int\limits_{\{h<|u-g|<h+1\}}\!\!\!\!\!\!\! |\nabla u|^p\,dxdt=0,
\ee
\begin{multline}\label{equg}
-\int_Q S(u-g)\vfi_t\,dxdt +\int_Q \adixt {\nabla u}\nabla ( S'(u-g)\vfi)\,dxdt
\\
= \int_Q fS'(u-g)\vfi\,dxdt+\int_Q G\nabla(S'(u-g)\vfi))\,dxdt+\into S(u_0)\vfi(0, x)\,dx 
\end{multline}
for every $S\in W^{2,\infty}(\RR)$ such that $S'$ has compact support, for every $\vfi\in C_c^\infty([0,T)\times\Omega)$.
\end{theorem}

\begin{proof} We split the proof in two steps.

\vskip0.3em
\noindent
\emph{Step 1.} Set $v=T_k(u)-g$. Then $v\in L^p(0,T;V)$. Moreover, using the  decomposition of $\mu$ in \rife{req}, and integrating by parts the term with $g_t$, we see that $v$ satisfies
\begin{multline*}
-\int_Q v\vfi_t\,dxdt +\int_Q \adixt {\nabla T_k(u)}\nabla \vfi\,dxdt
\\
= \int_Q f\vfi\,dxdt+ \int_Q G\nabla \vfi\,dxdt+\int_Q \vfi\,d\la_k+\into T_k(u_0)\vfi(0, x)\,dx 
\end{multline*}
for every $\vfi\in C_c^\infty([0,T)\times \Omega)$. It is easy to see that the above equality remains true for every $\vfi\in W^{1,\infty}(Q)$. Take then
$$
\vfi(x,t)= \xi(x,t) \frac1h\int_t^{t+h} \psi(v(s, x))ds,
$$
where $\xi\in C_c^\infty([0,T)\times \overline\Omega)$, $\xi\geq0$,  $\xi\psi(0)=0$ on $(0,T)\times \partial \Omega$, and $\psi$ is a  Lipschitz nondecreasing function. 
Since $\psi$ is nondecreasing we have (using \cite[Lemma~2.1]{BlPo})
\begin{multline*}
\liminf\limits_{h\to 0}\bigg\{ -\int_Q (v-T_k(u_0)) \Big(\xi  \frac1h\int_t^{t+h} \psi(v)ds \Big)_t\,dxdt 
\bigg\}\\
\geq  
-\int_Q \left(\int_0^v \psi(r)dr\right)\xi_t\,dxdt-
\into \left(\int_0^{T_k(u_0)} \psi(r)dr\right)\xi(0, x)\,dx.
\end{multline*}
Moreover, 
since $\psi$ is bounded we have
$$
\left|\int_Q \vfi\,d\la_k\right|\leq \|\xi\|_\infty \|\psi\|_\infty \|\la_k\|_{\pmisure}
$$
and since $\psi$ is Lipschitz  we have $\psi(v)\in \psob$, hence $(\psi(v))_h$ converges to $\psi(v)$
strongly in $\psob$ and weakly$^*$ in $\parelle\infty$. Therefore, we deduce, as $h\to 0$,
\begin{multline}\label{psi1}
-\int_Q \left(\int_0^v \psi(r)dr\right)\xi_t\,dxdt
+\int_Q \adixt {\nabla T_k(u)}\nabla (\psi(v)\xi)\,dxdt\\
\leq \int_Q f\psi(v)\xi\,dxdt
+ \int_Q G\nabla (\psi(v)\xi)\,dxdt\\
+ \into \bigg(\int_0^{T_k(u_0)} \psi(r)dr\bigg)\xi(0, x)\,dx + \|\xi\|_\infty \|\psi\|_\infty \|\la_k\|_{\pmisure},
\end{multline}
for every $\psi$ Lipschitz and nondecreasing.
In order to obtain the reverse inequality one can take 
$$
\vfi(x,t)= \xi(x,t) \frac1h\int_{t-h}^{t} \psi(\hat v(s, x))\, ds
$$ 
where $\hat v(x,t)= v(x,t)$ when $t\geq 0$ and $\hat v=U_j$ when $t<0$, being $U_j\in C_c^\infty(\Omega)$ such that $U_j\to T_k(u_0)$ strongly in $\elle1$. Using this time \cite[Lemma~2.3]{BlPo}
we obtain
\begin{multline*}
\liminf\limits_{h\to 0} \bigg\{-\int_Q (v-T_k(u_0))\Big(\xi  \frac1h\int_{t-h}^{t} \psi(v)ds\Big)_t\,dxdt \bigg\}\leq
\\
\leq  
-\int_Q \left(\int_0^v \psi(r)dr\right)\xi_t\,dxdt-
\into \left(\int_0^{U_j} \psi(r)dr\right)\xi(0, x)\,dx
\\
- \into (T_k(u_0)-U_j)\psi(U_j)\xi(0, x)\,dx.
\end{multline*}
Since it is still true that $\hat v\in \psob\cap\parelle\infty$, when $h\to 0$ we can pass to the limit in the other terms as above  and we obtain now
\begin{multline*}
-\int_Q \left(\int_0^v \psi(r)dr\right)\xi_t\,dxdt
+\int_Q \adixt {\nabla T_k(u)}\nabla (\psi(v)\xi)\,dxdt 
\\
\ge \int_Q f\psi(v)\xi\,dxdt+ \int_Q G\nabla (\psi(v)\xi)\,dxdt + \into \left(\int_0^{U_j} \psi(r)dr\right)\xi(0, x)\,dx \\
+ \into (T_k(u_0)-U_j)\psi(U_j)\xi(0, x)\,dx- \|\xi\|_\infty \|\psi\|_\infty \|\la_k\|_{\pmisure},
\end{multline*}
which implies, as $U_j\to T_k(u_0)$, that
\begin{multline}
\label{psi2}
-\int_Q \left(\int_0^v \psi(r)dr\right)\xi_t\,dxdt
+\int_Q \adixt {\nabla T_k(u)}\nabla (\psi(v)\xi)\,dxdt\\
 \geq \int_Q f\psi(v)\xi\,dxdt
+ \int_Q G\nabla (\psi(v)\xi)\,dxdt + \into \left(\int_0^{T_k(u_0)} \psi(r)dr\right)\xi(0, x)\,dx \\
- \|\xi\|_\infty \|\psi\|_\infty \|\la_k\|_{\pmisure}.
\end{multline}
Let $S\in W^{2,\infty}(\RR)$;  we use \eqref{psi1} with $\psi=\int_0^s (S''(t))^+dt$ and \eqref{psi2} with $\psi=\int_0^s (S''(t))^-dt$. Since $S'(s)=\int_0^s (S''(t)^+-S''(t)^-)dt$, subtracting the two inequalities we obtain
\begin{multline}\label{svk}
-\int_Q S(v)\xi_t\,dxdt
+\int_Q \adixt {\nabla T_k(u)}\nabla (S'(v)\xi)\,dxdt 
\\
\le \int_Q fS'(v)\xi\,dxdt+ \int_Q G\nabla (S'(v)\xi)\,dxdt \\
+ \into S(T_k(u_0))\xi(0, x)\,dx + 2\|\xi\|_\infty \|S'\|_\infty \|\la_k\|_{\pmisure},
\end{multline}
for every $S\in W^{2,\infty}(\RR)$ and for every nonnegative $\xi$. 

\vskip0.3em
\noindent
\emph{Step 2.}  Take in \eqref{svk} $S'= \theta_h(s)$, where $\theta_h=T_1(s-T_h(s))$, and $\xi=\xi(t)$. We obtain, denoting $ R_h(s)=\int_0^s \theta_h(\xi)d\xi$, 
\begin{multline*}
-\int_Q R_h(T_k(u)-g)\xi_t\,dxdt
+\int\limits_{\{ h<|u-g|<h+1\}} \!\!\!\!\!\!\!\!\adixt {\nabla T_k(u)}\nabla (T_k(u)-g)\xi\,dxdt
\\
\le \int_Q f\theta_h(T_k(u)-g)\xi\,dxdt+ \int\limits_{\{ h<|u-g|<h+1\}}  \!\!\!\!\!\!\!\!G\nabla (T_k(u)-g)) \xi \,dxdt \\
+ \into R_h(T_k(u_0))\xi(0, x)\,dx + 2\|\xi\|_\infty \|\la_k\|_{\pmisure},
\end{multline*}
which implies, using \eqref{coercp}, \eqref{cont} and Young's inequality,
\begin{multline*}
-\int_Q R_h(T_k(u)-g)\xi_t\,dxdt
+\int\limits_{\{ h<|u-g|<h+1\}} \!\!\!\!\!\!\!\!|\nabla T_k(u)|^p\xi\,dxdt
\\
\le
\int_Q f\theta_h(T_k(u)-g)\xi\,dxdt+  C \int\limits_{\{ h<|u-g|<h+1\}}  \!\!\!\!\!\!\!\! (|G|^{p'}+ |g|^p+ |b|^{p'}) \xi \,dxdt \\
+ \into R_h(T_k(u_0))\xi(0, x)\,dx + 2\|\xi\|_\infty \|\la_k\|_{\pmisure}.
\end{multline*}
Now let $k\to \infty$, thanks to \eqref{limnu} and Fatou's lemma we deduce
\begin{multline*}
-\int_Q R_h(u-g)\xi_t\,dxdt
+\int\limits_{\{ h<|u-g|<h+1\}} \!\!\!\!\!\!\!\!|\nabla u|^p\xi\,dxdt\\
 \leq \int_Q f\theta_h(u-g)\xi\,dxdt
+  C \int\limits_{\{ h<|u-g|<h+1\}}  \!\!\!\!\!\!\!\! (|G|^{p'}+ |g|^p+ |b|^{p'}) \xi \,dxdt\\
  + \into R_h(u_0)\xi(0, x)\,dx .
\end{multline*}
Choosing $\xi=1- \frac 1\vep T_\vep(t-\tau)^+$, for $\tau\in (0,T)$, and letting $\vep \to 0$,  leads to the estimate of $u-g$ in $\limitate1$. Similarly, the usual choice  of nonincreasing $\xi_\vep\in C^\infty_c([0,T))$ such that $\xi_\vep\to1$ allows to get
\begin{multline*}
\int\limits_{\{ h<|u-g|<h+1\}} \!\!\!\!\!\!\!\!|\nabla u|^p\,dxdt\\
\leq \int\limits_{\{|u-g|>h\}} \!\!\!\!\! |f|\,dxdt
+  C \int\limits_{\{ h<|u-g|<h+1\}}  \!\!\!\!\!\!\!\! (|G|^{p'}+ |g|^p+ |b|^{p'}) \xi \,dxdt  + \int\limits_{\{|u_0|>h\}}\!\!\!\! |u_0|\,dx ,
\end{multline*}
which implies \eqref{asyug}.
Now, let $S\in W^{2,\infty}(\RR)$ such that $S'$ has compact support, and  take a nonnegative  $\xi\in C^\infty_c([0,T)\times\Omega)$. Using now the regularity \eqref{regug}, we can pass to the limit in \eqref{svk} as $k\to\infty$, and thanks to \eqref{limnu} we obtain \eqref{equg}.
\end{proof}

\section{Equations with absorption}

  In this section we turn to the study of equations with absorption.   Let $h:\RR\mapsto\RR$ be a continuous function such that 
  \be\label{abs}
h(s)s\geq 0 \quad \text{for every $|s|\geq L$},
  \ee
  for some $L \ge 0$, and let us consider the evolution problem
\begin{equation}\label{main}
\begin{cases}
    u_t-\dive(a(t,x,\nabla u)) +h(u)=\mu & \text{in}\ Q,\\
    u=u_0  & \text{on}\ \{0\} \times \Omega,\\
 u=0 &\text{on}\ (0,T)\times\partial\Omega,
  \end{cases}
\end{equation} 
where $\mu\in \cM_0(Q)$, $u_0\in L^1(\Omega)$ and where $a : Q \times \RR^N \to \RR^N$ satisfies \eqref{coercp}--\eqref{monot}.  

As far as  the notion of  solution of \eqref{main} is concerned, we follow the definitions introduced in the above section. Namely,   a renormalized solution of \eqref{main} is a function $u$ such that $h(u)\in L^1(Q)$ and $u$ satisfies Definition~\ref{rendef} replacing $\mu$  with $\mu-h(u)$. Correspondingly, $u$ is a distributional solution if $h(u)\in L^1(Q)$ and $u$ is a distributional solution with $\mu-h(u)$ as right hand side.

First of all, observe that a straightforward modification of Theorem~\ref{contrazione} implies the following

\begin{theorem}\label{contrazione2}
 Let $u_1$, $u_2$ be two renormalized solutions of problem \eqref{main} with data $(u_{01}, \mu_1)$ and $(u_{02}, \mu_2)$ respectively. Then,
\begin{multline}\label{contra2}
\into (u_1-u_2)^+(\tau)\, dx + \int_0^\tau \into \left(h(u_1)-h(u_2)\right){\rm sign}^+(u_1-u_2)\,dxdt
\\
\leq \|(u_{01}-u_{02})^+\|_{\elle1}+  \| (\mu_1-\mu_2)^+\|_{\pmisure}
\end{multline}
for almost every $\tau\in (0,T)$. 
\end{theorem}

When $h$ is monotone, the above $L^1$-contraction principle plays a crucial role in such type of problems. 
We are now able to prove our main result concerning 
\eqref{main}.

\begin{theorem}\label{th2}
Let $\mu\in \cM_0(Q)$ and $u_0\in L^1 (\Omega)$ and let $h$ be a continuous function satisfying \eqref{abs}. Then, problem \eqref{main} admits a renormalized solution (in particular, a distributional solution).
If, in addition, $h$ is nondecreasing, then the renormalized solution is unique.
\end{theorem}

\begin{proof}
 As in the proof of Proposition~\ref{jn-app}, we take $\mu_{n}=\mu\ast\rho_n$, where $(\rho_n)$ is a sequence of mollifiers. We then consider the solutions $u_{n}$ of 
 \begin{equation}\label{5.1}
 \begin{cases}
    (u_{n})_{t}-\dive(a(t,x,\nabla \ujn))  +h(u_{n})=\mu_{n} & \text{in}\ Q,\\
    u_{n}=u_{0n} & \text{on}\ \{0\} \times \Omega,\\
 u_{n}=0 &\text{on}\ (0,T)\times\partial\Omega.
  \end{cases}
 \end{equation}
  Since $\| \mu_{n}\|_{\pmisure}\leq \|\mu\|_{\pmisure}$, using  assumption \eqref{abs} one easily gets that $(h(u_{n}))$ is bounded in $L^1(Q)$ and $\|h(u_n)\|_{\pelle1}\leq  \|\mu\|_{\pmisure}$.  Hence, the sequence $(u_{n})$ satisfies the estimates of Proposition~\ref{pro} and the compactness properties. In particular, there exists   $u\in L^1(Q)$ such that
\eqref{conver} holds up to a  subsequence.
Moreover, Theorem~\ref{stimcap} implies that
\be\label{equicap}
\lim\limits_{k\to \infty} \sup_{n} \capp \{ |u_{n}|>k \} =0.
\ee
Multiplying the equation in \eqref{5.1} by  $T_1( u_{n}-T_k(u_n))$ we obtain
$$
\int\limits_{\{|u_{n}| >k+1\}} |h(u_{n})| \leq \int\limits_{\{|u_{n}|>k\}} |\mu_{n}|.
$$
By Proposition~\ref{equi},  the sequence $(\mu_{n})$ is equidiffuse, so that we get
$$
\lim\limits_{k\to \infty}\sup\limits_n \int\limits_{\{|u_{n}| >k+1\}} |h(u_{n})|=0
$$
We now prove the equi-integrability of the sequence $(h(u_n))$.  Indeed, since for any subset $E \subset Q$ we have
$$
\sup\limits_n \int_E  |h(u_{n})| \leq  \sup\limits_n \int_E |h(T_{k+1}(u_{n}))|+  \sup\limits_n \int\limits_{\{|u_{n}| >k+1\}} |h(u_{n})|
$$
and since, for fixed $k$, the sequence $(h(T_{k+1}(u_{n}))$ is equi-integrable , we deduce that
$$
\lim\limits_{|E|\to 0}\sup\limits_n \int_E  |h(u_{n})| \leq \sup\limits_n \int\limits_{\{|u_{n}| >k+1\}} |h(u_{n})|.
$$
Hence letting $k\to \infty$ we get the equi-integrability of $(h(u_{n}))$.  Since $(h(u_n))$ converges pointwise to $h(u)$, by Vitali's theorem,
$$
h(u_{n}) \to h(u)\quad \hbox{strongly in $L^1(Q)$.}
$$
We can apply now Proposition~\ref{jn-app} to deduce that $u$ is a renormalized solution of \eqref{main}.
When $h$ is nondecreasing, we obtain uniqueness of the renormalized solution from Theorem~\ref{contrazione2}.
\end{proof}

In the case where $h$ is nondecreasing, the existence of a solution can also be proved in a  slightly different way, which consists in proving first the result for a dense subset of measures $\mu$, then using the $L^1$-contraction principle (Theorem~\ref{contrazione2}) to obtain the result for any diffuse measure $\mu$. 
We can take for instance the subset of measures satisfying the decomposition  \eqref{cap3c1} with $g\in L^\infty(Q)$; this set is dense in view of Theorem~\ref{app}. 
For such measures the existence of solutions of \eqref{main} can be proved in the lines of the elliptic case.

\section{Extension to the nonmonotone case}

The approach developed in this paper is not limited to the case that the divergence form operator is monotone.  Let for example $a : Q \times \re\times \rn \to \rn$ be a Carath\'eodory function (i.e., $a(\cdot,\cdot,s,\xi)$
is measurable on $Q$ for every $(s,\xi)$ in $\re\times\rn$, and $a(t,x,\cdot,\cdot)$ is
continuous on $\re\times\rn$ for almost every $(t,x)$ in $Q$) such that the
following holds:
\be
a(t,x,s,\xi) \cdot \xi \geq \al|\xi|^p,
\label{coercp2}
\ee
\be
|a(t,x,s,\xi)| \leq \beta[b(t,x) +|s|^{p-1} +|\xi|^{p-1}],
\label{cont2}
\ee
\be
[a(t,x,s,\xi) - a(t,x,s,\eta)] \cdot (\xi - \eta) > 0,
\label{monot2}
\ee
for almost every $(t,x)$ in $Q$, for every $s\in\re$ and for every $\xi$, $\eta$ in $\rn$, with
$\xi \neq \eta$, where, as before, $p > 1$, 
$\al$ and $\beta$ are two positive constants, and
$b$ is a nonnegative function in $L^{p'}(Q)$. From \eqref{coercp2} we can deduce that $a(x,t,s,0)=0$ for any $s\in \re$ and a.e. $(t,x)\in Q$. Consider the problem
\begin{equation}\label{nonlinuh}
\begin{cases}
    u_t-\dive(a(t,x,u,\nabla u))+ h(u)=\mu & \text{in}\ Q,\\
u=u_0  & \text{on}\ \{0\} \times \Omega,\\
u=0 &\text{on}\ (0,T)\times\partial\Omega,
  \end{cases}
\end{equation}
where $\mu$ is a diffuse measure, $u_0\in \elle1$ and $h$ satisfies \eqref{abs}. The  method  developed to find  existence of solutions relies on the possibility to find capacitary estimates. 
The proof we have given of such estimates in Theorem~\ref{stimcap} (and Theorem~\ref{stimcap2})  used  the monotone character of the second order term,  but we can generalize these estimates in  the following

\begin{theorem}\label{stimcap3}
Assume that \eqref{coercp2}--\eqref{monot2} hold. Given  $u_0\in \elle2$ and $\mu \in \mathcal{M}(Q)\cap\pw-1p'$, let $u\in W$ be a  (weak) solution of 
\begin{equation}\label{nonlinu2}
\begin{cases}
    u_t-\dive(a(t,x,u,\nabla u))=\mu & \text{in}\ Q,\\
u=u_0  & \text{on}\ \{0\} \times \Omega,\\
u=0 &\text{on}\ (0,T)\times\partial\Omega.
  \end{cases}
\end{equation}
Then,
$$
\capp(\{|u| > k\})\leq C\max\left\{\frac{1}{k^{\frac{1}{p}}},\frac{1}{k^\frac{1}{p'}}\right\} \quad \forall k\ge 1,
$$
where $C >0$ is a constant depending on $\|\mu\|_{\mathcal{M}(Q)}$, $\|u_0\|_{\elle1}$, $\|b\|_{L^{p'}(Q)}$, $\alpha$, $\beta$, $p$ and $\Omega$.
\end{theorem}
\begin{proof}
The strategy is the same of the proof of Theorem~\ref{stimcap} so we only sketch the main technical changes. Let us define the auxiliary function $\tilde{a}(t,x,\xi)\equiv a(t,x,u^+, \xi)$. Note that, due to \eqref{cont2},
$$
|\tilde{a}(t,x,\xi)|\leq \beta[b(t,x) +(u^+)^{p-1} +|\xi|^{p-1}]
$$
and since $u\in L^p(0,T;\sob)$ we have that $\tilde{a}(t,x,\xi)$ satisfies \eqref{cont} (and clearly also \eqref{coercp} and \eqref{monot}).
Since $\mu$ is diffuse then $\mu^+$ is diffuse as well;   we then consider the unique renormalized solution $v$ of the following problem 
\begin{equation}\label{pbau}
\begin{cases}
    v_t-{\rm div}(\tilde{a}(t,x,\nabla v))=\mu^+ & \text{in}\ Q,\\
v=u_0^+  & \text{on}\ \{0\} \times \Omega,\\
v=0 &\text{on}\ (0,T)\times\partial\Omega.
  \end{cases}
\end{equation}
It follows by Theorem~\ref{contrazione} (see Remark~\ref{comp-anom}) that  $ u^+\leq v$. 
Now, as in Proposition~\ref{tm} we can prove that 
$$
  T_k(v)_t-{\rm div}(\tilde{a}(t,x,\nabla T_k(v)))\geq 0\ \  \ \text{in}\ \ \mathcal{D}'(Q). 
$$
Without loss of generality, assume that $T_k(v(T))$ is well defined (otherwise use a sequence $t_n\uparrow T$  such that  $T_k(v(t_n))$ is so).  We  define the function $z\in W$ as the solution of 
\begin{equation*}\label{auxi}
\begin{cases}
    -z_t-{\rm div}(\tilde{a}(t,x,\nabla z)) =-2\, {\rm div}(\tilde{a}(t,x,\nabla T_k (v)))  & \text{in}\ Q,\\
 z=T_k (v)  & \text{on}\ \{T\} \times \Omega,\\
 z=0 &\text{on}\ (0,T)\times\partial\Omega, 
  \end{cases}
\end{equation*}
and we apply a comparison argument (since $-{\rm div}(\tilde{a}(t,x,\nabla z))$ is monotone) to deduce 
$$
z\geq T_k (v)\geq T_k (u^+)\qquad \hbox{a.e. in $Q$.}
$$
Now,  $T_k(v)$ satisfies the usual estimates \eqref{sti1} (e.g. by Proposition~\ref{ren-dis}) and since
\[
\begin{split}
|\tilde{a}(t,x,\nabla T_k (v))|^{p'}
& \leq C\big[b(t,x)^{p'} +(u^+)^{p} +|\nabla T_k (v)|^{p}\big] \,\chi_{\{v<k\}}
\\
& \leq C[b(t,x)^{p'} +(T_k(v))^{p} +|\nabla T_k (v)|^{p}]
,
\end{split}
\]
using Poincar\'e inequality we have
$$
\begin{array}{c}
\| \tilde{a}(t,x,\nabla T_k (v))\|_{L^{p'}(Q)}\leq C [ \|b\|_{L^{p'}(Q)}+
\| \nabla T_k(v)\|_{L^p(Q)}^{p-1}] \leq C(1+k^{\frac {p-1}p}).
\end{array}
$$
Therefore $z$ satisfies the estimate in $W$ as in Step 2 of Theorem~\ref{stimcap} and we  conclude that 
$$
\capp(\{u> k\})\leq C\max\left\{\frac{1}{k^{\frac{1}{p}}},\frac{1}{k^\frac{1}{p'}}\right\} \quad \forall k\geq 1.
$$
The estimate for $\{u<- k\}$ follows in the same way using $\mu^-$ and $u^-$. 
\end{proof}

Finally, once we have the capacitary estimates in hand, we can follow the proof of Theorem~\ref{th2} and, through approximation, we can find  the existence of at least one renormalized (in particular, weak) solution of \eqref{nonlinuh}.

\end{document}